\documentclass[a4paper]{article}

\usepackage[T1]{fontenc}
\usepackage{geometry}
\usepackage{amsmath}
\usepackage{amsthm}
\usepackage{amssymb}
\usepackage{color}
\usepackage{enumerate}
\usepackage[unicode=true,pdfusetitle,
 bookmarks=true,bookmarksnumbered=false,bookmarksopen=false,
 breaklinks=false,pdfborder={0 0 0},pdfborderstyle={},backref=false,colorlinks=false]
 {hyperref}

\makeatletter

\usepackage{appendix}

\usepackage{bm}

\usepackage[nameinlink,capitalise,noabbrev]{cleveref}
\crefname{appsec}{Appendix}{Appendices}

\hypersetup{%
    bookmarksnumbered, bookmarksopen=true, bookmarksopenlevel=1,%
}

\theoremstyle{plain}
\newtheorem{thm}{Theorem}[section]
\crefname{thm}{Theorem}{Theorems}
\theoremstyle{plain}
\newtheorem{lem}[thm]{Lemma}
\crefname{lem}{Lemma}{Lemmas}
\theoremstyle{plain}

\theoremstyle{plain}
\newtheorem*{claim*}{Claim}
\crefname{claim}{Claim}{Claims}
\theoremstyle{definition}
\newtheorem{defn}[thm]{Definition}
\theoremstyle{plain}
\newtheorem{conjecture}[thm]{Conjecture}
\crefname{conjecture}{Conjecture}{Conjectures}
\theoremstyle{plain}

\crefname{prop}{Proposition}{Propositions}
\theoremstyle{definition}

\theoremstyle{definition}

\theoremstyle{plain}
\newtheorem{claim}[thm]{Claim}

\crefformat{equation}{#2(#1)#3}
\crefname{appsec}{Appendix}{Appendices}

\date{}

\let\originalleft\left
\let\originalright\right
\renewcommand{\left}{\mathopen{}\mathclose\bgroup\originalleft}
\renewcommand{\right}{\aftergroup\egroup\originalright}
\usepackage{verbatim}

\makeatletter
\renewcommand*{\UrlTildeSpecial}{%
  \do\~{%
    \mbox{%
      \fontfamily{ptm}\selectfont
      \textasciitilde
    }%
  }%
}%
\let\Url@force@Tilde\UrlTildeSpecial
\makeatother

\usepackage{tikz}
\tikzstyle{vertex}=[circle,draw=black,fill=black,inner sep=0,minimum size=0.2cm,text=white,font=\footnotesize]
\tikzset{every loop/.style={min distance=50,in=50,out=130,looseness=7}}

\usepackage[labelfont=bf,labelsep=period]{caption}

\renewcommand{\textendash}{--}

\makeatother

\allowdisplaybreaks
\begin{document}
\title{Almost all Steiner triple systems are almost resolvable}
\author{Asaf Ferber\thanks{Department of Mathematics, University of California, Irvine.
Email: \href{mailto:asaff@uci.edu} {\nolinkurl{asaff@uci.edu}}.
Research supported in part by NSF grants DMS-1954395 and DMS-1953799.} \and Matthew Kwan\thanks{Department of Mathematics, Stanford University, Stanford, CA 94305.
Email: \href{mattkwan@stanford.edu}{\nolinkurl{mattkwan@stanford.edu}}.
Research supported in part by SNSF Project 178493 and NSF Award DMS-1953990.}}

\maketitle
\global\long\def\E{\mathbb{E}}%
\global\long\def\Var{\operatorname{Var}}%
\global\long\def\floor#1{\left\lfloor #1\right\rfloor }%
\global\long\def\ceil#1{\left\lceil #1\right\rceil }%

\begin{abstract}
We show that for any $n$ divisible by 3, almost all order-$n$ Steiner
triple systems admit a decomposition of almost all their triples into
disjoint perfect matchings (that is, almost all Steiner triple systems
are almost \emph{resolvable}).
\end{abstract}

\section{Introduction}

One of the oldest problems in combinatorics, posed by Kirkman in 1850~\cite{Kirkman 1} is the following: 

\emph{Fifteen young ladies in a school walk out three abreast seven days in succession: it is required to arrange them daily so that no two shall walk twice abreast.} 

This problem was solved by Kirkman a few years earlier~\cite{Kirkman} (although the first published solution is due to Cayley~\cite{Cayley}). 

In order to frame the above question in more generality we need to introduce some terminology. An \emph{order $n$ Steiner Triple System} ({\bf STS}$(n)$ for short) is a collection of triples of $[n]:=\{1,\ldots, n\}$ for which every pair of elements is contained in exactly one triple. These objects are named after Jakob Steiner, who observed the existence of such systems in 1853 (it is an interesting historical note that this happened after Kirkman proposed his problem). 

As an example, observe that the collection $\{123,145,167,246,257,347,356\}$
forms an {\bf STS}$(7)$, and a simple counting argument yields that an {\bf STS}$(n)$ can only exist if $n\equiv (1 \text{ or } 3) \mod 6$. It is also known that for every such $n$, a construction of an {\bf STS}$(n)$ can be achieved. The study of Steiner triple systems (and some natural generalisations, such as Steiner systems and block designs) has a long and rich history. These structures have strong connections to a wide range of different subjects, ranging from group theory, to finite geometry, to experimental design, to the theory of error-correcting codes, and more. For an introduction to the subject the reader is referred to \cite{CR99}.

Now, for an order-$n$ Steiner triple system $S$, a \emph{matching} in $S$ is a collection of disjoint triples of $S$,
and a \emph{perfect matching} (also known as a \emph{resolution class} or \emph{parallel class}) is a matching covering the entire ground set $[n]$. Observe that a perfect matching consists of exactly $\frac{n}{3}$ triples and therefore can only exist when $n=3 \mod 6$. We say that $S$ is \emph{resolvable} (or has \emph{parallelism}) if the triples in $S$ can be perfectly partitioned into perfect matchings. Since $S$ has $\binom{n}{2}/3$ triples and every perfect matching has $n/3$ edges, such a partition must consist of exactly $\frac{n-1}{2}$ perfect matchings.

Using the above terminology, Kirkman's problem is simply asking whether there exists a resolvable {\bf STS}$(15)$. More generally, one can ask whether there exists a resolvable {\bf STS}$(n)$ for any $n=3 \mod 6$. This problem was famously solved
in the affirmative by Ray-Chaudhuri and Wilson~\cite{RW71} in 1971,
over 100 years after Kirkman posed his problem.

Despite the difficulty in proving even that resolvable Steiner triple systems
exist, the existence of many large matchings seems to actually be
a ``typical'' property of Steiner triple systems. Early results
were based on R\"odl-nibble\footnote{The ``nibble'' is a type of probabilistic argument introduced by R\"odl~\cite{Rod85} for finding almost-perfect matchings in set systems. Far beyond the study of Steiner triple systems, it has had significant influence throughout combinatorics in the last 30 years.} type arguments. A result of Pippenger and Spencer~\cite{PS89} shows that every {\bf STS}$(n)$ contains disjoint matchings of size $n-o\left(n\right)$ which
cover a $\left(1-o\left(1\right)\right)$-fraction of its triples,
and Alon, Kim and Spencer~\cite{AKS97} later proved that every {\bf STS}$(n)$ contains at least one matching which covers
all but at most $O(n^{1/2}\log^{3/2}n)$ elements. More recently,
building on the breakthrough work of Keevash~\cite{Kee14} concerning
existence and completion of block designs, Kwan~\cite{Kwa16} proved that
if $n=3\mod 6$, then almost all (meaning a $\left(1-o\left(1\right)\right)$-fraction
of) order-$n$ Steiner triple systems have a perfect matching. In
fact, almost every {\bf STS}$(n)$ has $\left(\left(1-o\left(1\right)\right)n/\left(2e^{2}\right)\right)^{n/3}$
different perfect matchings. This proof was adapted by Morris~\cite{Mor17}
to show that almost every {\bf STS}$(n)$ has $\Omega\left(n\right)$
disjoint perfect matchings. We would go so far as to make the following conjecture:

\begin{conjecture}
\label{conj:kirkman}For $n=3\mod 6$, almost every {\bf STS}$(n)$ is resolvable.
\end{conjecture}

We remark that in \cref{conj:kirkman}, ``almost every'' certainly
cannot be replaced with ``every'': Bryant and Horsley~\cite{BH15}
proved that for infinitely many $n=3\mod 6$ there exist Steiner triple systems with not even a single perfect matching.

Of course, \cref{conj:kirkman} and the aforementioned result of Kwan
are really about\emph{ random }Steiner triple systems: we are trying
to understand properties that hold a.a.s.\footnote{By ``asymptotically almost surely'', or ``a.a.s.'', we mean that
the probability of an event is $1-o\left(1\right)$. Here and for
the rest of the paper, asymptotics are as $n\to\infty$.} in a uniformly random {\bf STS}$(n)$. This turns
out to be surprisingly difficult: despite the enormous advances in
the theory of random combinatorial structures since the foundational
work of Erd\H os and R\'enyi~\cite{ER59}, almost none of the available
tools seem to be applicable to random Steiner triple systems. Random
Steiner triple systems lack independence or any kind of recursive
structure, which rules out many of the techniques used to study Erd\H os\textendash R\'enyi
random graphs and random permutations, and there is basically no freedom
to make local changes, which precludes the use of ``switching''
techniques often used in the study of random regular graphs (see for
example \cite{KSVW01}). It is not even clear how to study a random
{\bf STS}$(n)$ empirically: in an attempt to find an efficient
algorithm to generate a random {\bf STS}$(n)$, Cameron~\cite{Cam02}
designed a Markov chain on Steiner triple systems, but he was not
able to determine whether this chain was connected. As far as we know,
before the work of Kwan~\cite{Kwa16}, the only nontrivial fact known
to hold for a random {\bf STS}$(n)$ was that it a.a.s.\ has trivial
automorphism group, a fact proved by Babai~\cite{Bab80} using a
direct (and rather coarse) counting argument.

Building on the ideas in \cite{Kwa16}, introducing the sparse
regularity method and extending a random partitioning argument from \cite{FKL17}, in this paper we prove the following ``asymptotic'' version of \cref{conj:kirkman},
adding to the short list of known facts about random Steiner triple
systems.
\begin{thm}
\label{conj:packing} Let ${\boldsymbol S}$ be a uniformly
random {\bf STS}$(n)$, where $n=3 \mod 6$. Then a.a.s.\ ${\boldsymbol S}$ has
$\left(1/2-o\left(1\right)\right)n$ disjoint perfect matchings.
\end{thm}

Our proof of \cref{conj:packing} is based on tools and intuition from random \emph{hypergraph} theory. We conclude this introduction with a brief discussion of random hypergraphs and some related work.

For $k\geq 2$, a $k$-uniform hypergraph $H$ (or a $k$-\emph{graph} for short) is a pair $H=(V,E)$, where $V$ is a finite set of \emph{vertices} and $E\subseteq \binom{V}{k}$ is a family of $k$-element subsets of $V$, referred to as \emph{edges}. The existence of perfect matchings is one of the most central questions in the theory of graphs and hypergraphs. In the case of graphs (that is, $2$-uniform hypergraphs), the problem of finding a perfect matching (if one exists) is relatively simple, but the analogous problem in the hypergraph setting is known to be NP-hard (see \cite{Karp}). Therefore, a main theme in extremal and probabilistic combinatorics is to investigate sufficient conditions for the existence of perfect matchings in hypergraphs. The results which are most relevant to our paper are those regarding the problem of finding perfect matchings in \emph{random} hypergraphs.

Let $\textrm{H}^k(n,p)$ be the (binomial) random $k$-uniform hypergraph distribution on the vertex set $V(H)=[n]$, where each $k$-set $X\in \binom{[n]}{k}$ is included as an edge with probability $p$ independently. A main problem in this area was to find a ``threshold'' function $q(n)$ for the property of containing a perfect matching. That is, to find a function $q$ for which, when $\boldsymbol G\sim \textrm{H}^k(n,p)$, we have

\[ \Pr\left[\boldsymbol G \text{ contains a perfect matching}\right]\rightarrow \left\{ \begin{array}{cc}
       1&  \text{if } p(n)/q(n) \rightarrow \infty   \\
       0&  \text{ if } p(n)/q(n) \rightarrow 0.\end{array} \right. \]

Following a long line of work, the most significant development in this area is the celebrated work of Johansson, Kahn and Vu~\cite{JKV08} where they showed that $q(n)=\log n/n$ defines a threshold function for this property. Two other relevant results are those of Frieze and Krivelevich \cite{FK} and its variant due to Ferber, Kronenberg and Long \cite{FKL17} which establish a way to find ``many'' edge-disjoint perfect matchings/Hamiltonian cycles in random (hyper)graphs based on some pseudorandom properties and a random partitioning argument. 

\subsection{Structure of the paper}

The structure of the paper is as follows. In \cref{sec:random-STS}
we review the methods introduced in \cite{Kwa16} for studying random
Steiner triple systems via the triangle removal process, and present
one new lemma for studying monotone decreasing properties. In \cref{sec:sufficient-properties}
we outline the general approach of the proof, stating some properties
that we will use to pack perfect matchings and that we will be able to
show hold a.a.s.\ in a random Steiner triple system.

Before turning to the proof of \cref{conj:packing}, in \cref{sec:sparse regularity}
we discuss the sparse regularity lemma for hypergraphs and some auxiliary
lemmas for applying it, and in \cref{sec:almost-perfect-matchings}
we discuss some lemmas for finding almost-perfect matchings in hypergraphs.

The proof of \cref{conj:packing} itself appears in \cref{sec:pack-in-good,sec:random-good}.
In \cref{sec:pack-in-good} we prove that certain properties suffice
for packing perfect matchings, and in \cref{sec:random-good} we show
that these properties a.a.s.\ hold in random Steiner triple systems.

Finally, in \cref{sec:concluding} we have some concluding remarks,
and in \cref{sec:KLR-proof} we explain how to generalise Conlon, Gowers,
Samotij and Schacht's proof \cite{CGSS14} of the so-called K\L R conjecture to ``linear'' hypergraphs. This will be a key tool in our proof.

\subsection{Notation}

We use standard asymptotic notation throughout, as follows. For functions
$f=f\left(n\right)$ and $g=g\left(n\right)$, we write $f=O\left(g\right)$
to mean that there is a constant $C$ such that $\left|f\right|\le C\left|g\right|$,
$f=\Omega\left(g\right)$ to mean that there is a constant $c>0$
such that $f(n)\ge c\left|g(n)\right|$ for sufficiently large $n$, $f=\Theta\left(g\right)$ to mean
that that $f=O\left(g\right)$ and $f=\Omega\left(g\right)$, and $f=o\left(g\right)$ to mean that $f/g\to0$ as $n\to\infty$.
Also, following \cite{Kee15}, the notation $f=1\pm\varepsilon$ means
$1-\varepsilon\le f\le1+\varepsilon$.

We also use standard graph theory notation: $V\left(G\right)$ and
$E\left(G\right)$ are the sets of vertices and (hyper)edges of a
(hyper)graph $G$, and $v\left(G\right)$ and $e\left(G\right)$ are
the cardinalities of these sets. The subgraph of $G$ induced by a
vertex subset $U$ is denoted $G\left[U\right]$, the degree of a
vertex $v$ is denoted $\deg_{G}\left(v\right)$, and the subgraph
obtained by deleting $v$ is denoted $G-v$. Given a subset of vertices $U\subseteq V(G)$ and any vertex $v\in V(G)$, we let $\deg_U(v)$ denote the degree of $v$ into the subset $U$ (that is, the number of edges consisting of $v$ and some vertices of $U$).

For a positive integer $n$, we write $\left[n\right]$ for the set $\left\{ 1,2,\dots,n\right\} $.
For a real number $x$, the floor and ceiling functions are denoted
$\floor x=\max\left\{ i\in{\mathbb Z}:i\le x\right\} $ and $\ceil x=\min\left\{ i\in{\mathbb Z}:i\ge x\right\} $.
We will however mostly omit floor and ceiling signs and assume large
numbers are integers, wherever divisibility considerations are not
important. We will use the convention that random objects (for example, random variables or random graphs) are printed in bold. Finally, all logarithms are in base $e$.

\section{Random Steiner triple systems via the triangle removal process\label{sec:random-STS}}

In this section we reproduce the general theorems from \cite{Kwa16}
for studying the behaviour of a randomly chosen {\bf STS}$(n)$ via the triangle removal
process, including a new lemma for studying monotone decreasing properties.
This machinery will be crucial to prove \cref{conj:packing}.

Note that any {\bf STS}$(n)$ is a $3$-graph and let $N={n \choose 2}/3=\left(1+o\left(1\right)\right)n^{2}/6$ be
the number of edges in any {\bf STS}$(n)$.
We assume throughout this section that $n$ is 1 or 3 mod 6 (as otherwise an {\bf STS}$(n)$ does not exist). Let us first make some useful definitions.

\begin{defn}[partial Steiner triple systems]
A \emph{partial Steiner triple system} (also known as a \emph{linear}
3-graph) is a $3$-graph on the vertex set $\left[n\right]$ in which every
pair of vertices is included in at most one edge. We will
also want to consider partial Steiner triple systems equipped with
an ordering on their edges. Let ${\mathcal O}$ be the set of ordered
Steiner triple systems, and let ${\mathcal O}_{m}$ be the set of ordered
partial Steiner triple systems with exactly $m$ edges. For $S\in{\mathcal O}_{m}$
and $i\le m$, let $S_{i}$ be the ordered partial Steiner triple
system consisting of just the first $i$ edges of $S$. For a
(possibly ordered) partial Steiner triple system $S$, let $G\left(S\right)$
be the graph (that is, a $2$-graph) with an edge for every pair of vertices which does not appear in any edge of $S$. So, if $S$ has $m$ edges,
then $G\left(S\right)$ has ${n \choose 2}-3m$ edges.
\end{defn}

\begin{defn}[quasirandomness]
For a graph $G$ with $n$ vertices and $m$ edges, let $d\left(G\right)=m/{n \choose 2}$
denote its density. We say $G$ is \emph{$\left(\varepsilon,h\right)$-quasirandom}
if for every set $A$ of at most $h$ vertices, we have $\left|\bigcap_{w\in A}N_{G}\left(w\right)\right|=\left(1\pm\varepsilon\right)d\left(G\right)^{\left|A\right|}n$.
Let ${\mathcal O}_{m}^{\varepsilon,h}\subseteq{\mathcal O}_{m}$ be the set of ordered
partial Steiner triple systems $S\in{\mathcal O}_{m}$ such that $G\left(S_{i}\right)$
is $\left(\varepsilon,h\right)$-quasirandom for each $i\le m$.
\end{defn}

\begin{defn}[the triangle removal process]
The triangle removal process is defined as follows. Start with the
complete graph $K_{n}$ and iteratively delete the edge-set of a triangle chosen uniformly
at random from all triangles in the remaining graph. If we continue
this process for $m$ steps, the deleted triangles (in order) can
be interpreted as an ordered partial Steiner triple system in ${\mathcal O}_{m}$.
It is also possible that the process aborts (because there are no
triangles left) before $m$ steps, in which case we say it returns
the value ``$*$''. We denote by ${\operatorname{R}}\left(n,m\right)$ the resulting
distribution on ${\mathcal O}_{m}\cup\left\{ *\right\} $.
\end{defn}

Now, we can state the general theorem from \cite{Kwa16} comparing
a typical {\bf STS}$(n)$ with a typical outcome of the triangle removal process. Basically,
if we can show that the first few edges of the triangle removal
process (as an ordered partial Steiner triple system) satisfy some
property with extremely high probability, then it follows that the
first few edges of a uniformly random ordered {\bf STS}$(n)$ satisfy the same property with high probability. Moreover,
it suffices to study the triangle removal process conditioned on some
``good'' event, provided that this event contains the event that
(the graph of uncovered edges of) our partial Steiner triple system
is sufficiently quasirandom.
\begin{thm}
\label{lem:triangle-removal-transfer}Fixing $h\in{\mathbb N}$
and sufficiently small $a>0$, there is $b=b\left(a,h\right)>0$ such
that the following holds. Fix $\alpha\in\left(0,1\right)$, let $\mathcal{P}\subseteq{\mathcal O}_{\alpha N}$
be a property of ordered partial Steiner triple systems, let $\varepsilon=n^{-a}$,
 let $\mathcal{Q}\supseteq{\mathcal O}_{\alpha N}^{\varepsilon,h}$, let ${\boldsymbol S}\in{\mathcal O}$
be a uniformly random ordered Steiner triple system and let ${\boldsymbol S}'\sim{\operatorname{R}}\left(n,\alpha N\right)$.
If
\[
\Pr\left({\boldsymbol S}'\notin\mathcal{P}\,\middle|\,{\boldsymbol S}'\in\mathcal{Q}\right)\le\exp\left(-n^{2-b}\right)
\]
then
\[
\Pr\left({\boldsymbol S}_{\alpha N}\notin\mathcal{P}\right)\le\exp\left(-\Omega\left(n^{1-2a}\right)\right).
\]
\end{thm}

\subsection{A coupling lemma\label{sec:nibble-TRP-coupling}}

The triangle removal process can be rather technical to study directly,
so \cite{Kwa16} included a general lemma approximating the first
few steps of the triangle removal process by a binomial random 3-graph
with relatively small edge probability. The idea is that instead of randomly
choosing triples one-by-one avoiding conflicts with previous choices,
one can randomly choose several triples at once, and just delete those
triples which conflict with each other. If the edge probability
is small, there are likely to be few conflicts, so these two processes (at least intuitively) give almost the same distribution.

This lemma in \cite{Kwa16} (specifically, \cite[Lemma~2.10]{Kwa16})
was suitable for studying properties $\mathcal{P}$ that are monotone
increasing in the sense that $S\in\mathcal{P}$ and $S'\supseteq S$
implies $S'\in\mathcal{P}$. Surprisingly, despite the fact that the
existence of perfect matchings (or collections of disjoint perfect
matchings) is a monotone increasing property, in this paper we will
instead need a similar lemma for monotone decreasing properties. Before stating the lemma we need the following definition.

\begin{defn}
For a partial Steiner triple system $S$, let ${\operatorname{R}}\left(S,m\right)$
be the partial Steiner triple system distribution obtained
with $m$ steps of the triangle removal process starting from $G\left(S\right)$. (So, if $S$ has $m'$ edges, then ${\operatorname{R}}\left(S,m\right)$ has $m+m'$ edges, unless the triangle removal process aborts).
\end{defn}

\begin{lem}
\label{lem:bite-transfer-new}Fix sufficiently small $\alpha\in\left(0,1\right)$.
Consider some $S\in{\mathcal O}_{m}^{\alpha,2}$ for some $m\le\alpha N$,
and consider a property $\mathcal{P}$ of unordered 3-graphs (which
may depend on $S$) that is monotone decreasing in the sense that
$G\in\mathcal{P}$ and $G'\subseteq G$ implies $G'\in\mathcal{P}$.
Let ${\boldsymbol S}\sim{\operatorname{R}}\left(S,\alpha N\right)$ and $\boldsymbol G\sim \operatorname{H}^3(n,q)$ with $q=\alpha(1+10\alpha)/n$.
Then
\[
\Pr\left({\boldsymbol S}\notin\mathcal{P}\right)\le\Pr\left(\boldsymbol G\notin\mathcal{P}\right)+e^{-\Omega\left(n^{2}\right)}.
\]
\end{lem}

The proof of \cref{lem:bite-transfer-new} is somewhat more complicated
than \cite[Lemma~2.10]{Kwa16}, but the intuition is basically the
same: if $\alpha$ is small and $S$ is a partial Steiner triple system
with few edges, then the outcome of $\operatorname{H}^3(n,q)$ is likely to ``almost'' be
a partial Steiner triple system and ``almost'' avoid conflicts
with $S$, and therefore approximates ${\operatorname{R}}\left(S,\alpha N\right)$. For the proof we will need the following concentration
inequality, which appears as \cite[Theorem~2.11]{Kwa16}, and is also a direct consequence of \cite[Theorem~1.3]{War16}. It is a
bounded-differences inequality with Bernstein-type tails which can
be used to analyse sparse random hypergraphs. Standard bounded-difference
inequalities such as the Azuma\textendash Hoeffding inequality do
not provide strong enough tail bounds to apply \cref{lem:triangle-removal-transfer}.
\begin{thm}
\label{thm:bernstein-type}Let $\boldsymbol{\omega}=\left(\boldsymbol{\omega}_{1},\dots,\boldsymbol{\omega}_{n}\right)$
be a sequence of independent, identically distributed random variables
with $\Pr\left(\boldsymbol{\omega}_{i}=1\right)=p$ and $\Pr\left(\boldsymbol{\omega}_{i}=0\right)=1-p$.
Let $f:\left\{ 0,1\right\} ^{n}\to{\operatorname{R}}$ satisfy the Lipschitz condition
$\left|f\left(\boldsymbol{\omega}\right)-f\left(\boldsymbol{\omega}'\right)\right|\le K$
for all pairs $\boldsymbol{\omega},\boldsymbol{\omega}'\in\left\{ 0,1\right\} ^{n}$
differing in exactly one coordinate. Then
\[
\Pr\left(\left|f\left(\boldsymbol{\omega}\right)-\E f\left(\boldsymbol{\omega}\right)\right|>t\right)\le\exp\left(-\frac{t^{2}}{4K^{2}np+2Kt}\right).
\]
\end{thm}

We will also need the fact that quasirandom graphs have approximately
the ``right'' number of triangles.
\begin{lem}
\label{lem:triangle-count}Let $\varepsilon$ be sufficiently small and let $G$ be an $\left(\varepsilon,2\right)$-quasirandom
graph with density $d$. Then $G$ has $\left(1\pm 3\varepsilon\right)d^{3}n^{3}/6$
triangles.
\end{lem}

\begin{proof}
By $\left(\varepsilon,1\right)$-quasirandomness, the sum of the degrees
is $\left(1\pm\varepsilon\right)dn^2$, so there are $\left(1\pm\varepsilon\right)dn^2/2$
edges. For each such edge, by $\left(\varepsilon,2\right)$-quasirandomness
there are $\left(1\pm\varepsilon\right)d^{2}n$ common neighbours
of that edge, each of which gives a triangle containing that edge.
Therefore the total number of triangles is $$\left(\left(1\pm\varepsilon\right)dn^2/2\right)\left(\left(1\pm\varepsilon\right)d^{2}n\right)/3=\left(1\pm3\varepsilon\right)d^{3}n^{3}/6$$
as desired. 
\end{proof}
Now we can prove \cref{lem:bite-transfer-new}.
\begin{proof}[Proof of \cref{lem:bite-transfer-new}]
Let ${\boldsymbol S}^{*}$ be obtained from ${\boldsymbol G}$ by deleting every edge
which intersects an edge of $S$ or another edge of ${\boldsymbol G}$
in two vertices. We can couple $\left({\boldsymbol S}^{*},{\boldsymbol G}\right)$ and
${\boldsymbol S}$ in such a way that ${\boldsymbol G}\supseteq{\boldsymbol S}$
as long as ${\boldsymbol S}^{*}$ has at least $\alpha N$ edges (randomly
order the edges in ${\boldsymbol G}$ and run the triangle removal process
with this ordering). Let ${\boldsymbol Y}$ be the number of edges in
${\boldsymbol S}^{*}$, which is the number of edges in ${\boldsymbol G}$ that
do not conflict with $S$ and are \emph{isolated} in the sense that
they do not intersect any other edge of ${\boldsymbol G}$ in more than
one vertex. By \cref{lem:triangle-count}, there are at least $\left(1\pm3\alpha\right)\left(1-\alpha\right)^{3}n^{3}/6=\left(1\pm7\alpha\right)n^{3}/6$
possibilities for such an edge, and each is present and isolated
with probability
\[
\left(\alpha\left(1+10\alpha\right)/n\right)\left(1-\alpha\left(1+10\alpha\right)/n\right)^{3\left(n-3\right)+1}=\alpha\left(1+10\alpha-o\left(1\right)\right)e^{-3\alpha\left(1+10\alpha\right)}/n=\left(1+ 7\alpha+O(\alpha^2)\right)\alpha/n
\]
for sufficiently small $\alpha$. This implies that $\E{\boldsymbol Y}=\alpha N+\Omega\left(n^{2}\right)$.
Next, observe that adding an edge to ${\boldsymbol G}$ can increase ${\boldsymbol Y}$
by at most 1, and removing an edge to ${\boldsymbol G}$ can increase ${\boldsymbol Y}$
by at most 3 (by making three edges isolated). So, by \cref{thm:bernstein-type},
\[
\Pr\left({\boldsymbol Y}<\alpha N\right)\le\Pr\left(\left|{\boldsymbol Y}-\E Y\right|\ge\Omega\left(n^{2}\right)\right)\le\exp\left(-\Omega\left(\frac{\left(n^{2}\right)^{2}}{3^{2}\binom{n}{3}\left(\alpha\left(1+10\alpha\right)/n\right)+3n^{2}}\right)\right)=e^{-\Omega\left(n^{2}\right)}.
\]
It follows that 
\[
\Pr\left({\boldsymbol S}\notin\mathcal{P}\right)\le\Pr\left({\boldsymbol G}\notin\mathcal{P}\right)+\Pr\left({\boldsymbol Y}<\alpha N\right)=\Pr\left({\boldsymbol G}\notin\mathcal{P}\right)+e^{-\Omega\left(n^{2}\right)}.\tag*{\qedhere}
\]
\end{proof}

\section{Sufficient properties for packing\label{sec:sufficient-properties}}

In this section we state some lemmas from which \cref{conj:packing}
will follow. Basically, the idea is to define some properties which
can be shown to hold a.a.s.\ in random Steiner triple systems, and
which can be used to find ``many'' perfect matchings. The first of these properties
is an ``upper-quasirandomness'' condition, which is a requirement
to effectively apply the so-called regularity method in the sparse setting.
\begin{defn}
For vertex sets $X,Y,Z$ in a 3-graph $G$, let $e_{G}\left(X,Y,Z\right)$
be the number of orderings $\left(x,y,z\right)$ of edges $\left\{ x,y,z\right\} $
with $x\in X,y\in Y,z\in Z$ (if $X,Y,Z$ are disjoint this is the
number of edges with exactly one vertex in each of $X,Y,Z$). A 3-graph
is $\left(p,\beta\right)$-\emph{upper-quasirandom} if for any vertex subsets
$X,Y,Z$, we have $e\left(X,Y,Z\right)\le p\left|X\right|\left|Y\right|\left|Z\right|+\beta n^{3}p$. \end{defn}

The second property we will need is the existence of certain special subgraphs which we refer to as \emph{absorbers}. We will later see that it is not very hard to construct ``almost''-perfect matchings, and we will be able to use the special features of absorbers to complete almost-perfect matchings into perfect matchings (by ``absorbing'' the uncovered vertices). This idea falls into the framework of the \emph{absorption method}, which was first introduced as a general method by R\"odl, Ruci\'nski and Szemer\'edi in \cite{RRS}, and has had numerous applications since.

\begin{defn}[sub-absorbers and absorbers]
\label{def:sub-absorber} 
The absorbers we use (see also \cref{fig:absorber} for an illustration) are based on \emph{sub-absorbers} and will be defined in two steps: 
\begin{itemize} 
\item A \emph{sub-absorber} rooted on a triple
of vertices $x,y,z$ is a set of five edges of the form
\[
\left\{ \left\{ x,x_{1},x_{2}\right\} ,\left\{ y,y_{1},y_{2}\right\} ,\left\{ z,z_{1},z_{2}\right\} ,\left\{ x_{1},y_{1},z_{1}\right\} ,\left\{ x_{2},y_{2},z_{2}\right\} \right\} .
\]
We call $x,y,z$ the \emph{rooted vertices }of the sub-absorber and
we call the other nine vertices the \emph{external vertices}. If an edge contains a rooted vertex, we call it a \emph{rooted edge}

\item An \emph{absorber }rooted on a triple of vertices $x,y,z$ is a set
of $13$ edges obtained in the following way. Put three disjoint
edges $\left\{ x,x_{1},x_{2}\right\} ,\left\{ y,y_{1},y_{2}\right\} ,\left\{ z,z_{1},z_{2}\right\} $,
then put a sub-absorber rooted on $\left\{ x_{1},y_{1},z_{1}\right\} $
and a sub-absorber rooted on $\left\{ x_{2},y_{2},z_{2}\right\} $
(in such a way that no pair of edges intersects in more than
one vertex). We call $x,y,z$ the \emph{rooted vertices} of the absorber
and the other $18$ vertices the \emph{external vertices}. If an edge contains a rooted vertex, we call it a \emph{rooted edge}. Note that
the edges of an absorber can be partitioned into two matchings:
a perfect matching with seven edges (which we call the \emph{covering
matching}), and a matching with six edges (which we call the
\emph{non-covering matching}).
\end{itemize}
\end{defn}

\begin{figure}[h]
\begin{center}
\includegraphics{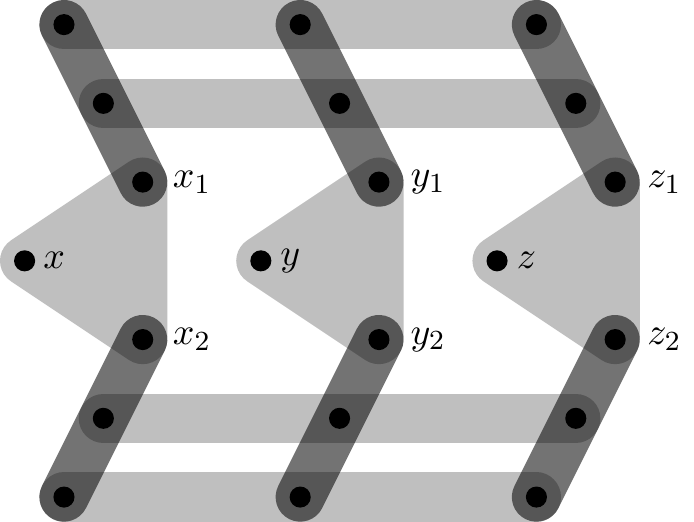}
\end{center}

\caption{\label{fig:absorber}An illustration of an absorber for $\left(x,y,z\right)$.
The light edges are the covering matching and the dark edges
are the non-covering matching.}
\end{figure}

The crucial property of an absorber, that allows it to be used to complete a matching into
a perfect matching, is the existence of the covering and non-covering
matching (we can choose whether to cover the root vertices $x,y,z$
or not). The observant reader may notice that sub-absorbers themselves
are simpler structures also having a covering and non-covering matching;
the reason we consider the larger absorbers is that it is easier to
find a rooted absorber in a random 3-graph than a rooted sub-absorber
(mainly due to the fact that the rooted edges in an absorber
have no common incident edges).

In the proof of our main result we will want to construct our perfect matchings one-by-one, so it will
be important that we can find absorbers disjoint to a set of perfect
matchings that have already been constructed. In order to do so, we will show that the $3$-graphs we are working with are ``resilient'' with respect to absorbers in the sense that every ``dense'' subgraph contains at least one absorber for each triple.

\begin{defn}
A $3$-graph $G$ is \emph{$D$-absorber-resilient} if every subgraph
$G'\subseteq G$ with minimum degree at least $D$ has an absorber
rooted on every triple $x,y,z$ of distinct vertices of $G$.
\end{defn}

Now, the properties we will use for packing perfect matchings are as follows.  It will be useful for the reader to think of $\beta$ as a very small constant which controls
the error terms, and $\alpha$ as a constant that is much larger (but
still quite small in absolute terms) that measures the fraction of
a Steiner triple system we are considering.

\begin{defn}
Consider $\alpha,\beta>0$, and let $S$ be any $3$-graph. We say that $S$ is \emph{$\left(\alpha,\beta\right)$-good
}if it satisfies the following properties.
\begin{enumerate}
\item {\bf Almost-regularity}: for every vertex $v$ we have $\deg_{S}\left(v\right)=\alpha n/2\pm\beta n$;
\item {\bf Pseudorandomness}: $S$ is $\left(\alpha/n,\beta\right)$-upper-quasirandom;
\item {\bf Robust absorber-resilience}: for each $w\ge\beta n$, letting $\eta=w/n$, all but at most $$\exp\left(-10^{-8}(w/n)^4\alpha^2n\right)\binom{n}{w}$$
of the $w$-vertex induced subgraphs $S\left[W\right]$ are $0.999\eta^{2}\alpha\left(n/2\right)$-absorber-resilient.
\end{enumerate}
\end{defn}

Our proof strategy for finding many perfect matchings will roughly go as follows. Suppose that ${\boldsymbol S}$ is a typical {\bf STS}$(n)$. We show that (the edge-set of) ${\boldsymbol S}$ can be decomposed into $(\alpha,\beta)$-good subgraphs, for some carefully chosen $\alpha,\beta$. Then, in each of these subgraphs, we show that one can approximately decompose its edges into perfect matchings. Together, we obtain our desired collection of $(1-o(1))n/2$ perfect matchings.

One might wonder why we need this intermediate stage of partitioning ${\boldsymbol S}$, instead of working with ${\boldsymbol S}$ directly. The reason we do so
is that ``small bits'' of ${\boldsymbol S}$ behave like the binomial random hypergraph model (recall \cref{lem:bite-transfer-new}) so we can borrow some tools and ideas from this well-studied model in order to construct our matchings.

Now we state our key lemmas, which together imply \cref{conj:packing}.

\begin{lem}
\label{conj:random-good} There is $r\in \mathbb{N}$ such that the following holds. Suppose $n$ is congruent to $1$ or $3$ mod $6$ and let ${\boldsymbol S}$ be a uniformly random {\bf STS}$(n)$. Then 
a.a.s. the edge-set of ${\boldsymbol S}$ can be partitioned into $r$ spanning subgraphs that are $\left(1/r,o\left(1\right)\right)$-good.
\end{lem}

\begin{lem}
\label{conj:pack-in-good}For any $\alpha=\Omega\left(1\right)$ and
$n\equiv0\operatorname{mod}3$, every $n$-vertex $\left(\alpha,o\left(1\right)\right)$-good
linear 3-graph $S$ (that is, a partial Steiner triple system) has $\left(\alpha/2-o\left(1\right)\right)n$
disjoint perfect matchings.
\end{lem}

The proof of \cref{conj:random-good} mainly comes down to studying
the robust absorber-resilience property of goodness. (The almost regularity and pseudorandomness properties can be proved in a fairly straightforward fashion
using \cref{lem:triangle-removal-transfer,lem:bite-transfer-new}). For the robust absorber-resilience property, we will use \cref{lem:triangle-removal-transfer,lem:bite-transfer-new}
in combination with a hypergraph generalisation of the sparse regularity lemma (to be stated in \cref{sec:sparse regularity}) and a hypergraph generalisation of the resolution of the K\L R
conjecture by Conlon, Gowers, Samotij and Schacht (to be stated in \cref{sec:sparse regularity} as well). Roughly speaking,
this allows us to reduce certain problems about subgraphs of a random
hypergraph to problems about \emph{dense} hypergraphs, and the latter case is much more tractable to study. One complication is that even though robust absorber-resilience 
is a property of subgraphs of a random hypergraph, an absorber is a \emph{rooted} structure. We will need some additional tricks to reduce
the situation to one where the K\L R conjecture can actually be applied. The
details of the proof will be presented in \cref{sec:random-good}.

Concerning \cref{conj:pack-in-good}, one may wonder how the robust absorber-resilience
property of goodness could be strong enough to produce an approximate decomposition of ${\boldsymbol S}$ into perfect matchings. Na\"ively, it seems that after we have
found only $0.001\alpha n/2$ perfect matchings, this property no
longer gives us any guarantee for the existence of absorbers in the
remaining edges. Indeed, the proof of \cref{conj:pack-in-good}
would be much simpler if the robust absorber-resilience property of goodness could be strengthened
to guarantee that ${\boldsymbol S}$ is $D$-absorber-resilient for $D=o\left(n\right)$.
Unfortunately, it is not clear how to prove that such a strong property
holds in a typical {\bf STS}$(n)$, even with the K\L R conjecture in hand.
Instead, we employ a random partitioning trick inspired by the work of
Ferber, Kronenberg and Long~\cite{FKL17}. We partition the edges
of ${\boldsymbol S}$ into many subgraphs with different roles, some of which are
used to find almost-perfect matchings and some of which are used to
complete almost-perfect matchings into perfect matchings (using absorbers).
The latter subgraphs contain only a small fraction of the edges
of ${\boldsymbol S}$, but have comparatively high degree (this is possible because
each of these subgraphs have a very small number of vertices). We
then only need a weak absorber-resilience property for these smaller
subgraphs. The details of the proof will be presented in \cref{sec:pack-in-good}.

\section{\label{sec:sparse regularity}Sparse regularity and the K\L R conjecture
for hypergraphs}

Kohayakawa and R\"odl~\cite{KR03} proved a sparse version of the so-called regularity lemma
for graphs. In this paper, for the proofs of \cref{conj:random-good,conj:pack-in-good}
we will need a generalisation of the sparse regularity lemma to hypergraphs. Fortunately, while the general hypergraph regularity lemma is much more complicated to prove (and state) than the graph regularity lemma, for our purposes (embedding \emph{linear} hypergraphs) we only need a sparse version of the ``weak'' hypergraph regularity lemma (see for example \cite[Theorem~9]{KNRS10}). To state our sparse regularity lemma for hypergraphs we need a few
definitions.

\begin{defn}
Let $\varepsilon,\eta>0$ and $0\leq p\leq 1$ be arbitrary parameters. 
\begin{itemize} 
\item {\bf Density:} Consider disjoint vertex sets $X_1,\dots,X_r$ in an $r$-graph $G$. Let $e(X_1,\ldots,X_r)$ be the number of edges with a vertex in each $X_i$, and let $$d(X_1,\ldots,X_r)=\frac{e(X_1,\ldots,X_r)}{|X_1|\dots |X_r|}$$ be the \emph{density} between $X_1,\dots,X_r$.

\item {\bf Regular tuples:} An $r$-graph $G$ is said to be $r$-\emph{partite} if its vertex set $V(G)$ consists of a partition $V(H)=V_1\cup \ldots \cup V_r$ into $r$ \emph{parts} in such a way that each of its edges intersect each $V_i$ in exactly one vertex. An $r$-partite $r$-graph with parts $V_{1},\dots,V_{r}$ is \emph{$\left(\varepsilon,p\right)$-regular}
if, for every $V_{1}'\subseteq V_{1},\dots,V_{r}'\subseteq V_{r}$
with $\left|V_{i}'\right|\ge\varepsilon\left|V_{i}\right|$, the density
$d\left(V_{1}',\dots,V_{r}'\right)$ of edges between $V_{1}',\dots,V_{r}'$
satisfies
\[
\left|d\left(V_{1}',\dots,V_{r}'\right)-d\left(V_{1},\dots,V_{r}\right)\right|\le\varepsilon p.
\]

\item {\bf Regular partitions:} A partition of the vertex set of a $r$-graph into $t$ parts $V_{1},\dots,V_{t}$
is said to be $\left(\varepsilon,p\right)$-regular if it is
an equipartition, and for all but at most $\varepsilon t^{r}$ $r$-tuples
$\left(V_{i_{1}},\dots,V_{i_{r}}\right)$, the induced $r$-partite
$r$-graph between $V_{i_{1}},\dots,V_{i_{r}}$ is $\left(\varepsilon,p\right)$-regular.

\item {\bf Upper-uniformity:} An $r$-graph $G$ is \emph{$\left(\eta,p,D\right)$-upper-uniform} if for
any choice of disjoint subsets $U_{1},\dots,U_{r}$ with $\left|U_{1}\right|,\dots,\left|U_{r}\right|\ge\eta\left|V\left(G\right)\right|$,
we have $d\left(U_{1},\dots,U_{2}\right)\le Dp$.
\end{itemize}
\end{defn}

Note that upper-uniformity is a weaker property than upper-quasirandomness:
if a 3-graph $G$ is $\left(o\left(1\right),p\right)$-upper-quasirandom,
then it is $\left(o\left(1\right),p,1+o\left(1\right)\right)$-upper-uniform.
Now, our hypergraph regularity lemma is as follows. We omit its proof since it is straightforward to adapt a proof of the sparse graph regularity lemma (see \cite{KR03} for a sparse regularity lemma for graphs, and see \cite[Theorem~9]{KNRS10} for a weak regularity lemma for dense hypergraphs).
\begin{lem}
\label{lem:sparse-regularity}For every $\varepsilon,D>0$ and every
positive integer $t_{0}$, there exist $\eta>0$ and $T\in{\mathbb N}$ such
that for every $p\in\left[0,1\right]$, every $\left(\eta,p,D\right)$-upper-uniform
$r$-graph $G$ with at least $t_{0}$ vertices admits an $\left(\varepsilon,p\right)$-regular
partition $V_{1},\dots,V_{t}$ of its vertex set into $t_{0}\le t\le T$
parts.
\end{lem}

We will use \cref{lem:sparse-regularity} for several different purposes.
Crucial to all of these is the notion of a \emph{cluster hypergraph}
which is a dense hypergraph which encodes the large-scale structure
of a regular partition. From now on we return to considering only
$3$-graphs.
\begin{defn}
Given an $\left(\varepsilon,p\right)$-regular partition $V_{1},\dots,V_{t}$
of the vertex set of a 3-graph $G$, the \emph{cluster hypergraph} is the 3-graph whose vertices are the clusters $V_{1},\dots,V_{t}$,
with an edge $\left\{ V_{i},V_{j},V_{k}\right\} $ if $d\left(V_{i},V_{j},V_{k}\right)>2\varepsilon p$
and the induced tripartite 3-graph between $V_{i}$, $V_{j}$ and
$V_{k}$ is $\left(\varepsilon,p\right)$-regular.
\end{defn}

If the regularity lemma is applied with small $\varepsilon$ and large
$t_{0}$, the cluster hypergraph approximately inherits certain density
properties from the original graph $G$, as follows.
\begin{lem}
\label{lem:transfer-cluster-graph}Fix sufficiently small $\varepsilon>0$
and sufficiently large $t_{0}\in{\mathbb N}$, and let $G$ be an $n$-vertex
$\left(p,o\left(1\right)\right)$-upper-quasirandom 3-graph. Let $G'\subseteq G$
be a spanning subgraph with minimum degree at least $\delta p\binom{n}{2}$.
Let $\mathcal{R}$ be the $t$-vertex cluster 3-graph obtained by
applying the sparse regularity lemma to $G'$ with parameters $t_{0}$, $p$
and $\varepsilon$. Then all but $\sqrt{\varepsilon}t$ vertices of
$\mathcal{R}$ have degrees at least $\left(\delta-3\sqrt{\varepsilon}-3/t_{0}\right)\binom{t}{2}$.
\end{lem}

The proof of \cref{lem:transfer-cluster-graph} is by a standard counting argument. It is a special case of \cref{lem:transfer-cluster-graph-partition}, to appear in the next subsection (so we defer the proof until then).

An immediate application of \cref{lem:transfer-cluster-graph} is the
fact that high-degree subgraphs of upper-quasirandom 3-graphs have
a ``rich'' vertex subset $Z$ such that most vertices outside $Z$
have reasonably high degree into $Z$. This will be important for
the proof of \cref{conj:pack-in-good}.

\begin{lem}
\label{lem:rich-set}For any $\delta>0$ and any $\varepsilon>0$
that is sufficiently small relative to $\delta$, there is $\xi>0$
such that the following holds. Let $G$ be an $\left(p,o\left(1\right)\right)$-upper-quasirandom
3-graph. Let $G'\subseteq G$ be a spanning subgraph with minimum
degree $\delta p\binom{n}{2}$. Then there is a $2\varepsilon n$-vertex
subset $U\subseteq V\left(G\right)$ such that the following conditions
hold.
\begin{enumerate}
\item For all but at most $2\sqrt{\varepsilon}n$ vertices $v$, there
are at least $\xi\binom{n}{2}p$ edges of $G'$ containing $v$
and two vertices of $U$;
\item every subset of $\left|U\right|-\xi n$ vertices of $U$ induces at least one edge of $G'$. \end{enumerate}
\end{lem}

\begin{proof}
Apply the sparse regularity lemma (\cref{lem:sparse-regularity}) to $G'$ with
small $\varepsilon$ and large $t_{0}$, and let $U$ contain a $2\varepsilon$-fraction
of each cluster. Let $\mathcal{R}$ be the cluster 3-graph (with $t\le T$
vertices).

By \cref{lem:transfer-cluster-graph}, if $\varepsilon$ is small enough
and $t_{0}$ is large enough, all but $\sqrt{\varepsilon}t$ clusters
$V_{i}$ have degree at least $\left(\delta/2\right)\binom{t}{2}$
in the cluster graph, and in particular participate in an $\left(\varepsilon,p\right)$-regular
triple $V_{i},V_{j},V_{q}$ with $d\left(V_{i},V_{j},V_{q}\right)>2\varepsilon p$ (recall that we are assuming that $\varepsilon$ is small relative to $\delta$).

Now, take $\xi=\varepsilon\left(1/T^{2}\right)$. To verify the first part of \cref{lem:rich-set}, we observe that for any cluster $V_i$ as above (participating in an $\left(\varepsilon,p\right)$-regular
triple $V_{i},V_{j},V_{q}$ with $d\left(V_{i},V_{j},V_{q}\right)>2\varepsilon p$), at least a $(1-\varepsilon)$-fraction of the vertices $v\in V_i$ satisfy the condition in the first part of \cref{lem:rich-set}.  Indeed, let $V_{j}'=V_{j}\cap U$ and $V_{q}'=V_{q}\cap U$, and let $V_{i}'$
be the set of vertices $v\in V_{i}$ for which there are fewer than
$\varepsilon\left|V_{j}'\right|\left|V_{j}'\right|p$ edges containing
$v$, a vertex in $V_{j}'$ and a vertex in $V_{j}'$. Then, $d\left(V_{i}',V_{j}',V_{q}'\right)<\varepsilon p$,
and if we had $\left|V_{i}'\right|\ge\varepsilon\left|V_{i}\right|$
this would contradict $\left(\varepsilon,p\right)$-regularity. So, $\left|V_{i}'\right|<\varepsilon\left|V_{i}\right|$. Then, observe that $\varepsilon\left|V_{j}'\right|\left|V_{j}'\right|p\ge \xi \binom{n}{2}p$.

For the second property, if we delete fewer than $\xi n$ vertices
from $U$ then we still have an $\varepsilon$-fraction of each cluster
$V_{i}$ and therefore have at least one edge.
\end{proof}

\subsection{Refining an existing partition}

In our proof of \cref{conj:random-good} we will apply the sparse regularity
lemma to a 3-graph whose vertices are already partitioned into a few
different parts with different roles. It will be important that the
regular partition in \cref{lem:sparse-regularity} can be chosen to
be consistent with this existing partition.
\begin{lem}
\label{lem:respect-partition}Suppose that a 3-graph $G$ has its
vertices partitioned into sets $U_{1},\dots,U_{h}$. In the $\left(\varepsilon,p\right)$-regular
partition guaranteed by \cref{lem:sparse-regularity}, we can assume
that all but at most $\varepsilon ht$ of the clusters $V_{i}$ are
contained in some $U_{j}$.
\end{lem}

For the reader who is familiar with the proof of the regularity lemma, the proof of \cref{lem:respect-partition} is straightforward. Indeed, in order to prove the regularity lemma, one starts with an arbitrary partition and keeps refining it in a clever way. Nevertheless, for completeness we include a short reduction from \cref{lem:sparse-regularity}.

\begin{proof}[Proof of \cref{lem:respect-partition}]
Apply \cref{lem:sparse-regularity} with regularity parameter $\varepsilon^{2}/2$.
For each cluster $V_{i}$, order the vertices in $V_{i}$ according
to the partition $U_{1},\dots,U_{h}$ (first the vertices from $U_{1}$,
then from $U_{2}$, etc). Now, equipartition $V_{i}$ into $\floor{1/\varepsilon}$
intervals $V_{i}^{1},\dots,V_{i}^{\floor{1/\varepsilon}}$ with respect
to this ordering. At most $h$ of these intervals intersect multiple
$U_{j}$, and the $V_{i}^{q}$ form an $\left(\varepsilon,p\right)$-regular
partition. (Strictly speaking some of the clusters may now have sizes
differing by 2 instead of 1, but we can move some vertices between
clusters to correct this without having any material effect on the
regularity of the partition).
\end{proof}

We also need a more technical version of \cref{lem:transfer-cluster-graph}
translating the degrees between the $U_{i}$ into degrees in the cluster
graph. First, we generalise the definition of a cluster graph.
\begin{defn}
Given a 3-graph $G$ with vertex partition $U_{1},\dots,U_{h}$ and
an $\left(\varepsilon,p\right)$-regular partition $V_{1},\dots,V_{t}$
of its vertices, the \emph{partitioned} cluster graph $\mathcal{R}$
with \emph{threshold }$\tau$ is the 3-graph defined as follows. The
vertices of $\mathcal{R}$ are the clusters $V_{i}$ which are completely
contained in some $U_{j}$, with an edge $\left\{ V_{i},V_{j},V_{k}\right\} $
if $d\left(V_{i},V_{j},V_{k}\right)>\tau p$ and the induced tripartite
3-graph between $V_{i}$, $V_{j}$ and $V_{k}$ is $\left(\varepsilon,p\right)$-regular.
\end{defn}

\begin{lem}
\label{lem:transfer-cluster-graph-partition}Fix sufficiently small
$\varepsilon>0$ and sufficiently large $t_{0}\in{\mathbb N}$, and let $G$
be an $n$-vertex $\left(p,o\left(1\right)\right)$-upper-quasirandom
3-graph with a partition $U_{1},\dots,U_{h}$ of its vertices into
parts of sizes $u_{1},\dots,u_{h}$, respectively. Let $G'\subseteq G$ be a spanning
subgraph such that each $v\in U_{i}$ has $\deg_{U_{j}}\left(v\right)\ge\delta_{ij}p\binom{u_{j}}{2}$ (where the degree here is with respect to $G'$).

Let $\mathcal{R}$ be the partitioned cluster 3-graph with threshold
$\tau$ obtained by applying \cref{lem:sparse-regularity,lem:respect-partition}
to $G'$ with parameters $t_{0}$, $p$ and $\varepsilon$. Let $t_{i}$ be
the number of clusters contained in $U_{i}$ and let $\mathcal{U}_{i}$
be the set of clusters contained in $U_{i}$. Then each $t_{i}\ge\left(u_{i}/n\right)t-\varepsilon ht$
and for each $i,j$, all but $\sqrt{\varepsilon}t$ clusters $X\in\mathcal{U}_{i}$
have $\deg_{\mathcal{U}_{j}}\left(X\right)\ge\delta_{ij}\binom{t_{i}}{2}-t^{2}\left(\tau+2\varepsilon h+\sqrt{\varepsilon}+2/t_{0}\right)$
in the cluster graph $\mathcal{R}$.
\end{lem}
Note that \cref{lem:transfer-cluster-graph} is actually a special
case of \cref{lem:transfer-cluster-graph-partition} (taking the trivial
partition and threshold $2\varepsilon$).
\begin{proof}[Proof of \cref{lem:transfer-cluster-graph-partition}]
The clusters in $\mathcal{U}_{i}$ comprise at most $t_{i}\left(n/t\right)$
vertices, so $\left|U_{i}\right|=u_{i}\le t_{i}\left(n/t\right)+\varepsilon h t \left(n/t\right)$.
It follows that $t_{i}\ge\left(u_{i}/n\right)t-\varepsilon ht$ as
claimed.

Let $\mathcal{W}$ be the set of all clusters $W$ in the $\left(\varepsilon,p\right)$-regular
partition for which there are more than $\sqrt{\varepsilon}\binom{t}{2}$
non-$\left(\varepsilon,p\right)$-regular triples $\left(V_{i},V_{j},W\right)$
containing $W$. There can be at most $3\varepsilon\binom{t}{3}/\left(\sqrt{\varepsilon}\binom{t}{2}\right)\le\sqrt{\varepsilon}t$
clusters in $\mathcal{W}$.
Let $Z$ be the set of at most $\varepsilon hn$
vertices whose cluster does not appear in the cluster 3-graph (because
the cluster was not completely contained in any $U_{i}$).

If a cluster $X\in\mathcal{U}_{i}\backslash\mathcal{W}$ has $\deg_{\mathcal{U}_{j}}\left(X\right)=d$
in the cluster 3-graph, then by $\left(p,o\left(1\right)\right)$-upper-quasirandomness
the number of edges of $G'$ with a vertex in $X$ and two vertices
in $U_{j}$ is at most
\begin{align*}
 & d\left(1+o\left(1\right)\right)p\left(\frac{n}{t}\right)^{3}+\sqrt{\varepsilon} \binom{t}{2}p\left(\frac{n}{t}\right)^{3}+\tau \binom{t_{j}}{2}p\left(\frac{n}{t}\right)^{3}+e_{G}\left(X,Z\cup X,U_{j}\right)+\sum_{W\in V\left(\mathcal{R}\right)}e_{G}\left(X,W,W\right)\\
 & \qquad\le\left(1+o\left(1\right)\right)p\left(\frac{n}{t}\right)^{3}\left(d+t^{2}\left(\tau+\sqrt \varepsilon\right)+t\left(\varepsilon ht+1\right)+t\right).
\end{align*}
But by the degree assumption in $G'$ and the fact that $u_{j}\ge t_{j}\left(n/t\right)$ this
number is at least 
\[
\left(n/t\right)\delta_{ij}p\binom{u_{j}}{2}\ge p\delta_{ij}\binom{t_{j}}{2}\left(\frac{n}{t}\right)^{3}.
\]
It follows that 
\[
d\ge\delta_{ij}\binom{t_{j}}{2}-t^{2}\left(\tau+2\varepsilon h+\sqrt{\varepsilon}+2/t\right).
\]
\end{proof}

\subsection{The K\L R conjecture}
One of the most powerful aspects of the sparse regularity method is
that, for a subgraph $G$ of a typical outcome of a random graph, if we find a substructure in the cluster graph (which is usually
dense, therefore comparatively easy to analyse), then a corresponding
structure must also exist in the original graph $G$. For graphs,
this was conjectured to be true by Kohayakawa, \L uczak and R\"odl~\cite{KLR97},
and was proved by Conlon, Gowers, Samotij and Schacht~\cite{CGSS14}.
We will need a generalisation to hypergraphs, which again we state
in more general form than we need, in case it is useful for future
applications. First we need some definitions.
\begin{defn}
Consider an $r$-graph $H$ with vertex set $\left\{ 1,\dots,k\right\} $
and let $\mathcal{G}\left(H,n,m,p,\varepsilon\right)$ be the collection
of all $r$-graphs $G$ obtained in the following way. The vertex set
of $G$ is a disjoint union $V_{1}\cup\dots\cup V_{k}$ of sets of
size $n$ . For each edge $\left\{ i_{1},\dots,i_{r}\right\} \in E\left(H\right)$,
we add to $G$ an $\left(\varepsilon,p\right)$-regular $r$-graph
with $m$ edges between $V_{i_{1}},\dots,V_{i_{r}}$. These are
the only edges of $G$.
\end{defn}

\begin{defn}
For $G\in\mathcal{G}\left(H,n,m,p,\varepsilon\right)$, let $\#_H(G)$
be the number of ``canonical copies'' of $H$ in $G$, meaning that
the copy of the vertex $i$ must come from $V_{i}$.
\end{defn}

\begin{defn}
\label{def:3-density}The $r$-\emph{density} $m_{r}\left(H\right)$
of an $r$-graph $H$ is defined as
\[
m_{r}\left(H\right)=\max\left\{ \frac{e\left(H'\right)-1}{v\left(H'\right)-r}:H'\subseteq H\text{ with }v\left(H'\right)>r\right\} .
\]
\end{defn}

Now, the hypergraph version of the K\L R conjecture is as follows.

\begin{thm}
\label{thm:KLR}For every linear $r$-graph $H$ (that is, an $r$-graph where every two edges intersect on at most one vertex) and every $d>0$,
there exist $\varepsilon,\xi>0$ with the following property. For
every $\eta>0$, there is $C>0$ such that if $p\ge CN^{-1/m_{r}\left(H\right)}$,
then with probability $1-e^{-\Omega\left(N^{r}p\right)}$ the following
holds in $\boldsymbol{G}\in \operatorname{H}^r(N,p)$. For every $n\ge\eta N$,
$m\ge dpn^{r}$ and every subgraph $G'$ of $\boldsymbol{G}$ in $\mathcal{G}\left(H,n,m,p,\varepsilon\right)$,
we have 
\[
\#_H(G')>\xi\left(\frac{m}{n^{r}}\right)^{e\left(H\right)}n^{v\left(H\right)}.
\]
\end{thm}

The proof of \cref{thm:KLR} is almost exactly the same as the proof
of \cite[Theorem~1.6(i)]{CGSS14}. In \cref{sec:KLR-proof} we will
describe the exact changes one needs to make in order to turn the proof in \cite{CGSS14}
into a proof of \cref{thm:KLR}.

\section{\label{sec:almost-perfect-matchings}Almost-perfect matchings}

For the proof of \cref{conj:pack-in-good} we will need multiple different
ways to find almost-perfect matchings, which we will then be able
to complete into perfect matchings using absorbers. The first result we will need is that high-degree subgraphs of upper-quasirandom
3-graphs have almost-perfect matchings.
\begin{lem}
\label{lem:almost-matching}Let $G$ be an $\left(p,o\left(1\right)\right)$-upper-quasirandom
3-graph. Let $G'\subseteq G$ be a spanning subgraph with minimum
degree at least $0.9p\binom{n}{2}$. Then $G'$ has a matching covering
all but $o\left(n\right)$ vertices.
\end{lem}

\begin{proof}
For sufficiently large $n'$, every $n'$-vertex 3-graph with minimum
degree at least $0.8\binom{n'}{2}$ has a perfect matching; see for
example \cite{HPS09}. So by \cref{lem:transfer-cluster-graph}, if
we apply the sparse regularity lemma (\cref{lem:sparse-regularity})
to $G'$ with small $\varepsilon$ and large $t_{0}$, we can find a matching
covering $t-2\sqrt{\varepsilon}t$ vertices of the cluster graph.
In each corresponding triple of clusters $V_{i},V_{j},V_{q}$ we can
greedily find a matching with $\left(1-\varepsilon\right)\left(n/t\right)$
vertices, and we can combine these to get a matching in $G'$ covering
$n-3\sqrt{\varepsilon}n$ vertices. Since $\varepsilon$ was arbitrary,
this implies that we can find a matching covering all but $o\left(n\right)$
vertices.
\end{proof}

The second result we need is that almost-regular 3-graphs can be almost-partitioned
into almost-perfect matchings, and moreover the leftover vertices
in each matching can be assumed to be ``well-distributed''.
\begin{thm}
\label{thm:stronger-PS}Fix $\alpha\in\left[0,1\right]$, and consider
a linear 3-graph $S$ with all degrees $\alpha n\pm o\left(n\right)$.
Then $S$ has $\alpha n-o\left(n\right)$ edge-disjoint matchings
$M_{1},\dots,M_{\alpha n-o\left(n\right)}$, each with $n/3-o\left(n\right)$
edges, such that every vertex of $S$ appears in all but $o\left(n\right)$
of the $M_{i}$.
\end{thm}

\cref{thm:stronger-PS} is a simple consequence of the following theorem of Pippenger and Spencer~\cite{PS89}, proved using a R\"odl-nibble-type
argument, which we reproduce below.
\begin{thm}
\label{thm:P-S}Fix $\alpha\in\left[0,1\right]$, and consider a linear
3-graph $S$ with all degrees $\alpha n\pm o\left(n\right)$. Then
the edges of $S$ can be partitioned into $\alpha n+o\left(n\right)$ edge-disjoint matchings.
\end{thm}

\begin{proof}[Proof of \cref{thm:stronger-PS}]
Fix any $\varepsilon>0$, which we treat as constant. Consider a partition of edge-disjoint
matchings as guaranteed by \cref{thm:P-S}. Note that each vertex appears in $\alpha n+o(n)$ of the matchings, by the almost-regularity condition on $S$. Let $Q$ be the number of matchings with fewer than $n/3-\varepsilon n$ edges, so that the total number of edges covered by all the matchings is $e(S)\le Q(n/3-\varepsilon n)+(\alpha n+o(1)-Q)(n/3)=\alpha n^2/3-Q\varepsilon n+o(n^2)$. But by the degree condition, we have $e(S)=\alpha n^2/3+o(n^2)$, so $Q=o(n)$.

Now, deleting the $Q$ matchings with fewer than $n/3-\varepsilon n$ edges, we obtain a collection of $\alpha n+o(n)-Q=\alpha n-o(n)$ edge-disjoint matchings each with at least $n/3-\varepsilon n$ edges, such that each vertex appears in $\alpha n+o(n)-Q$ of the matchings, which is all but $o(n)$ of them. Since $\varepsilon$ could have been taken arbitrarily small, the desired result follows.
\end{proof}

\section{\label{sec:pack-in-good}Packing in good systems}

In this section we prove \cref{conj:pack-in-good}. First, we show
how to partition our 3-graph into subgraphs with certain
``nice'' properties.

\subsection{Partitioning for packing}

The majority of the edges of our 3-graph $S$ will go into subgraphs
$G_{1},\dots,G_{\ell}$. These subgraphs will have vertex sets $U_{1},\dots,U_{\ell}$
which each comprise almost all the vertices of $S$, but they will be rather
sparse (each containing approximately a $1/\ell$ fraction of the
edges of $S$). Eventually, we will use \cref{thm:stronger-PS}
to find an almost-perfect packing of almost-perfect matchings in each
of these subgraphs.

Some of the remaining edges will go into subgraphs $F_{1},\dots,F_{\ell}$,
where the vertex set $W_{i}$ of each $F_{i}$ is complementary to
the vertex set of $G_{i}$. Despite each $F_{i}$ having fewer edges
than $G_{i}$, the degrees in $F_{i}$ will be much higher than the
degrees in $G_{i}$ (this will be possible because the $W_{i}$ will
be quite small, and $F_{i}$ will contain almost all the edges
of $S$ within $W_{i}$).

Many of the edges that still remain will go into subgraphs $H_{1},\dots,H_{\ell}$,
whose purpose is to serve as a ``bridge'' between $G_{i}$ and $F_{i}$.
For each $i\le\ell$, after finding our almost-perfect matchings in
$G_{i}$, we will use $H_{i}$ to extend each matching to cover all
of $G_{i}$ and some of $F_{i}$, after which we can iteratively complete
all of our matchings using absorber-resilience properties of $F_{i}$. The details
of the properties we will need are summarised in the following lemma. Say that a graph is $(R,D)$-\emph{absorber-resilient} if it is $D$-absorber-resilient after deleting any choice of at most $R$ vertices.
\begin{lem}
\label{thm:partitioning}For any fixed $\alpha$,
consider an $\left(\alpha,o\left(1\right)\right)$-good linear 3-graph
$S$ with $n$ vertices. Fix any (sufficiently small) $\delta>0$,
and let $\ell=\delta^{-5/2}$. Then there exists a constant $\kappa=\kappa(\alpha,\delta)>0$ and a partition of the edges of $S$ into $3\ell+1$ subgraphs $G_{1},\dots,G_{\ell}$, $H_{1},\dots,H_{\ell}$,
$F_{1},\dots,F_{\ell}$, and $Q$ (not necessarily induced or spanning)
such that the following properties hold:
\begin{enumerate}
\item {\bf Most of the edges are covered:} at least a $\left(1-3\sqrt{\delta}\right)$-fraction of the edges
of $S$ are in some $G_{i}$;
\item {\bf Controlling the sizes of $U_i$ and $W_i$:} for each $i$, the vertex sets $U_{i}=V\left(G_{i}\right)$ and $W_{i}=V\left(F_{i}\right)$
partition $V\left(S\right)$, and $\left|W_{i}\right|=\delta n+o\left(n\right)$;
\item {\bf The 3-graphs $G_i$ are almost regular:} each $G_{i}$ has all degrees in the range $$\left(\alpha\left(1-\delta\right)^{2}\left(1-2\sqrt{\delta}\right)/\ell\right)n/2\pm o\left(n\right);$$
\item {\bf The 3-graphs $F_i$ are relatively dense:} each $F_{i}$ has all degrees at least $0.9999\left(\alpha\delta^{2}\right)n/2$;
\item {\bf Many ``bridging'' edges:} for every vertex of $G_{i}$, there are $\Omega\left(n\right)$ edges
of $H_{i}$ with one vertex in $U_{i}$ and two vertices in $W_{i}$;
\item {\bf The 3-graphs $F_i$ are absorber-resilient:} each $F_{i}$ is $\left(\kappa n,0.9995\left(\alpha\delta^{2}\right)\left(n/2\right)\right)$-absorber-resilient.
\end{enumerate}
\end{lem}

Before we prove \cref{thm:partitioning} we briefly state and prove a lemma regarding absorber-resilience.

\begin{lem}
\label{lem:resilience-union-bound}For any fixed $\alpha,\delta>0$,
consider an $\left(\alpha,o\left(1\right)\right)$-good linear 3-graph
$S$ with $n$ vertices. There is $\kappa=\kappa(\alpha,\delta)>0$ such that if $U$ is a random subset of vertices of $S$, obtained by including each vertex with probability $\delta$ independently, then $S[U]$ is $\left(\kappa n,0.9995\left(\alpha\delta^{2}\right)\left(n/2\right)\right)$-absorber-resilient with probability at least $1-e^{-\Omega(n)}$.
\end{lem}

\begin{proof}
Choose $\kappa$ sufficiently small such that  $\binom{n}{\kappa n}\exp\left(-10^{-8}(\delta/2)^4\alpha^2 n\right)=e^{-\Omega(n)}$. By the Chernoff bound, we have $|U|\ge \delta n/2$ with probability at least $1-e^{-\Omega(n)}$. Condition on such an outcome for $|U|$, so that $U$ is now a uniformly random set of vertices of this size. Now, the desired result follows from the robust-absorber resilience property of goodness, taking the union bound over all ways to delete up to $\kappa n$ vertices from $U$.
\end{proof}

Now we can prove \cref{thm:partitioning}.

\begin{proof}[Proof of \cref{thm:partitioning}]
We will describe a random procedure to build the $G_{i},H_{i},F_{i}$
and show that the desired properties are satisfied with positive probability.
It suffices to show that each of the six properties holds with probability strictly larger than (say) $5/6$ individually; the result will then follow by a simple union bound.

For each $i\le\ell$ and each $v\in V\left(S\right)$, put $v\in W_{i}$
with probability $\delta$ (independently for each $i,v$). Then,
let $U_{i}=V\left(S\right)\backslash W_{i}$. By a simple application of the Chernoff bounds we see that  a.a.s.\ property 2 holds. Moreover, by \cref{lem:resilience-union-bound} we obtain that a.a.s.\ each $S\left[W_{i}\right]$ is $\left(\kappa n,0.9995\left(\alpha\delta^{2}\right)\left(n/2\right)\right)$-absorber-resilient. Since we will choose $F_{i}$ to be a spanning subgraph of $S\left[W_{i}\right]$,
we will obtain that a.a.s.\ property 6 holds.

Now, we build the $G_{i}$, $H_{i}$, $F_{i}$. Do the following for
each edge $e$, independently.
\begin{enumerate}
\item If $e$ is a subset of some $W_{i}$ (which happens with probability
$p_{F}:=1-\left(1-\delta^{3}\right)^{\ell}\approx\delta^{3}\ell=\sqrt{\delta}$),
then do the following.
\begin{enumerate}
\item If $e$ is a subset of a \emph{unique} $W_{i}$, put $e\in F_{i}$;
\item otherwise, put $e\in Q$.
\end{enumerate}
\item Choose a uniformly random $i\le\ell$, and choose $p_{H}$ to satisfy
$p_{F}+p_{H}=2\sqrt{\delta}$. If $e$ is not a subset of any $W_{j}$,
then with probability $p_{H}/\left(1-p_{F}\right)$ do the following.
\begin{enumerate}
\item If $e$ has one vertex in $U_{i}$ and two vertices in $W_{i}$, put
$e$ in $H_{i}$.
\item Otherwise, put $e\in Q$.
\end{enumerate}
\item The probability we have not taken any of the previous actions is $p_{G}:=1-2\sqrt{\delta}$.
In this case, do the following (still with $i\le\ell$ as a uniformly
random index).
\begin{enumerate}
\item If $e\subseteq U_{i}$, put $e$ in $G_{i}$.
\item Otherwise, put $e$ in $Q$.
\end{enumerate}
\end{enumerate}
By the Chernoff bound, a.a.s.\ property 1 holds.

Now, for a vertex $v$, let $d_{G_i}(v)$ be the number of edges $e\in S$ containing $v$, such that $e\setminus \{v\}\subseteq U_i$ and such that in the above procedure, the random index chosen for $e$ is $i$. So, if $v\in U_i$ then $d_{G_i}=\deg_{G_i}(v)$. Since $S$ is a partial Steiner triple system, the edges containing $v$ do not intersect other than in $v$, so $d_{G_i}(v)$ has a binomial distribution $\operatorname{Bin}\left(\alpha n/2\pm o\left(n\right),\left(1-\delta\right)^{2}p_{G}/\ell\right)$. By the Chernoff bound and the union bound it follows that a.a.s.\ property 3 holds.

Next, if for some $i\le\ell$ and edge $e\in E\left(S\right)$
we condition on the event that $e\subseteq W_{i}$ then the probability
that $e$ is contained in some other $W_{j}$ is $p^{*}:=1-\left(1-\delta^{3}\right)^{\ell-1}\approx\sqrt{\delta}$.
So, for every $i\le\ell$ and vertex $v$, if we condition on the
event that $v\in W_{i}$ then $\deg_{F_{i}}\left(v\right)$ has a
binomial distribution $\operatorname{Bin}\left(\alpha n/2\pm o\left(1\right),\delta^{2}\left(1-p^{*}\right)\right)$,
so by the Chernoff bound a.a.s.\ property 4 holds (provided $\delta$
is sufficiently small).

Finally, for every $i\le\ell$ and vertex $v$, if we condition on the
event that $v\in U_{i}$ then $\deg_{F_{i}}\left(v\right)$ has a binomial distribution $\operatorname{Bin}\left(\Omega(n),\Omega(1)\right)$, so by the Chernoff bound a.a.s.\ property 5 holds.
\end{proof}

\subsection{Absorbers}

Absorbers are the basic building blocks for a larger structure which
will eventually allow us to complete an almost-perfect matching into
a perfect matching. The relative positions of the absorbers in this
structure will be determined by a ``template'' with a ``resilient
matching'' property.
\begin{lem}
\label{lem:resilient-template}For any sufficiently large $n$, there
exists a 3-graph $T$ with $10n$ vertices, maximum degree at most
40, and an identified set $Z$ of $2n$ vertices, such that if we
remove any $n$ vertices from $Z$, the resulting hypergraph has a
perfect matching. We call $T$ a \emph{resilient template} and we
call $Z$ its \emph{flexible set.}
\end{lem}

\cref{lem:resilient-template} is an immediate consequence of \cite[Lemmas~5.2 and 5.3]{Kwa16},
and is proved using a construction due to Montgomery~\cite{Mon14}.
Now, we will want to arrange absorbers in the positions prescribed
by a resilient template, as follows.
\begin{defn}
An \emph{absorbing structure} is a 3-graph $H$ of the following form.
Consider a resilient template $T$ and put externally vertex-disjoint
absorbers on each edge of $T$, introducing 18 new vertices for
each. Then delete the edges of $T$. That is, the template just
describes the relative positions of the absorbers, its edges
are not actually in the absorbing structure.
\end{defn}

Note that an absorbing structure with a flexible set of size $2n$
has at most $10n+18\times400n/3=O\left(n\right)$ vertices, at most
$13\times\left(400n/3\right)=O\left(n\right)$ edges and maximum
degree at most $40=O\left(1\right)$. An absorbing structure $H$
has the same crucial property as the resilient template $T$ that
defines it: if we remove any half of the vertices of the flexible set
$Z$ then what remains of $H$ has a perfect matching. Indeed, after this
removal we can find a perfect matching $M$ of $T$, then our perfect
matching of $H$ can be comprised of the covering matching of the
absorber on each edge of $M$ and the non-covering matching for
the absorber on each other edge of $T$.

So, if we can find an absorbing structure $H$ with flexible set $Z$
in our 3-graph $S$, then to find a perfect matching it suffices to
find a matching that covers all the vertices outside $H$ and any
half of the vertices in $Z$. For the proof of \cref{conj:pack-in-good},
we will be able to construct an absorbing structure with prescribed
flexible set using absorber-resilience and the following simple lemma.
\begin{lem}
\label{lem:embed-absorbing-structure}Suppose an $n$-vertex 3-graph
$G$ has the property that in every induced subgraph with $n-n'$
vertices, there is an absorber rooted on every triple of vertices.
Suppose also that $n'$ is sufficiently large. Then given any subset
$Z\subseteq V\left(G\right)$ of size  $\omega(1)=|Z|\leq  n'/\left(18\times400/3\right)$, $G$ contains an absorbing structure with flexible set $Z$.
\end{lem}

\begin{proof}
Let $q=\left|Z\right|/2$, and fix a resilient template $T$ on an
arbitrary set of $10q$ vertices of $G$ (the edges of $T$ do
not have to exist in $G$). Now, we can build our absorbing structure
greedily, iteratively finding disjoint absorbers on each edge
of $T$. Indeed, at any point, the non-rooted vertices of the absorbers
found so far together comprise a total of at most $18\times400q/3\le n'$
vertices. After deleting these vertices, we can still continue to
find absorbers rooted on every desired triple of vertices.
\end{proof}
In the proof of \cref{conj:pack-in-good}, to find a perfect matching
we will choose $Z$ to be a ``rich'' set of vertices as guaranteed
by \cref{lem:rich-set}. Then, we will be able to use \cref{lem:almost-matching}
to find a matching covering almost all the vertices outside our absorbing
structure, and the choice of $Z$ will allow us to extend this to
a matching covering all the vertices outside the absorbing structure
and half of $Z$. We will then use the absorbing structure to complete
this into a perfect matching.

\subsection{Proof of \texorpdfstring{\cref{conj:pack-in-good}}{\cref{conj:pack-in-good}}}

Now, we combine the lemmas in the last two subsections to prove \cref{conj:pack-in-good}.
\begin{proof}[Proof of \cref{conj:pack-in-good}]
Fix sufficiently small $\delta>0$ (which we will treat as constant
for most of the proof). Consider subgraphs $G_{1},\dots,G_{\ell}$,
$H_{1},\dots,H_{\ell}$, $F_{1},\dots,F_{\ell}$ and vertex sets $U_1,\dots,U_\ell,W_1,\dots,W_\ell$ as guaranteed by
\cref{thm:partitioning}.

Consider some $i\le\ell$. Let $M_{1},\dots,M_{s}$ be a collection
of $s=\left(\alpha\left(1-\delta\right)^{2}\left(1-2\sqrt{\delta}\right)/\ell\right)n/2-o\left(n\right)$
almost-perfect matchings in $G_{i}$, as guaranteed by \cref{thm:stronger-PS}.
We will show that for each $j\le s$, regardless of how $M_{1},\dots,M_{j-1}$
were previously completed, we can complete $M_{j}$ to a perfect matching
using the edges in $H_{i}$ and $F_{i}$ that have not been used
so far. We will be able to conclude that $S$ has $\alpha\left(1-\delta\right)^{2}\left(1-2\sqrt{\delta}\right)n/2-o\left(n\right)$
edge-disjoint perfect matchings, which implies the desired result
(since $\delta$ could have been chosen to be arbitrarily small).

Now, fix an arbitrary ordering of the $o(n)$ vertices in $U_i$ which are not yet covered by $M_{j}$. For each such uncovered vertex $u$ in turn, choose an edge of $H_i$ which contains $u$ and two vertices of $W_i$, which has not already been used for a previous completion and which does not intersect any edges chosen for previous uncovered vertices. Add this edge to $M_j$. To see that it is possible to make this choice, note that there are $\Omega(n)$ edges in $H_i$ which contain $u$. Only $o(n)$ of these edges have been used to complete previous matchings, and only $o(n)$ of these edges intersect an edge that was previously chosen for a different uncovered vertex (the edges containing $u$ do not intersect in any vertices other than $u$, because $S$ is a partial Steiner triple system).

After doing this, $M_{j}$ covers all of $U_{i}$ and $o\left(n\right)$ vertices
of $W_{i}$. Let $W_{i}'$ be the set of unmatched vertices in $W_i$, and let
$F_{i}'$ be the subgraph consisting of all edges of $F_{i}\left[W_{i}'\right]$
not used for previous matchings $M_{q}$, $q<j$. We now need to find a perfect matching
in $F_{i}'$.

First note that if $\delta$ is sufficiently small, then $s$ (which
is about $\alpha\delta^{5/2}n/2$) is much less than the degrees in
$F_{i}$ (which are about $\alpha\delta^{2}n/2$). So, $F_{i}$ and
$F_{i}'$ have minimum degree at least $0.9998\left(\alpha\delta^{2}\right)n/2$. Let $\kappa=\Omega(1)$ be as in the statement of \cref{thm:partitioning}, let $n'=\min(\kappa n/2,0.0003(\alpha \delta^2)n/2)$, and note that by the absorber-resilience property of $F_i$ (property 6 in \cref{thm:partitioning}), $F_{i}'$ has the property that in any induced subgraph
obtained by deleting at most $n'$
vertices, there is an absorber rooted on every triple of vertices.

By \cref{lem:rich-set} (with sufficiently small $\varepsilon$) and
the pseudorandomness property of goodness (which implies the necessary upper-quasirandomness
condition), for some $\xi=\Omega\left(1\right)$ we can find a vertex
set $Z'$ with the following properties: 

\begin{enumerate}
    \item $\Omega\left(n\right)\le\left|Z'\right|\le n'/\left(2\times 18\times400/3\right)$;
    \item all but $0.01\alpha\delta^{2}n$ vertices have degree $\Omega\left(n\right)$
into $Z'$;
\item  every subset of $\left|Z'\right|\left(1-\xi\right)$
vertices of $Z'$ induces at least one edge.
\end{enumerate}
Arbitrarily add vertices
to obtain a set $Z\supseteq Z'$ such that $\left|Z\right|/2=\left|Z'\right|\left(1-\xi\right)$.
By the absorber-resilience properties of $F_i'$ and \cref{lem:embed-absorbing-structure},
we can find an absorbing structure $H$ in $F_{i}'$ with flexible
set $Z$. Let $X=V\left(H\right)\backslash Z$.

Consider the (at most $0.01\alpha\delta^{2}n$) bad vertices which
do not have the guaranteed degree into $Z'$. Since there are so few
of these vertices, and the absorbing structure $H$ is so small compared
to the degrees in $F_{i}'$, we can greedily find a matching in $F_{i}'$
avoiding $V\left(H\right)$ and covering all the bad vertices. Let
$F_{i}''$ be obtained from $F_{i}'$ by removing the matched vertices
and all the vertices in $X\cup Z'$. Then by \cref{lem:almost-matching}
we can find a matching covering all but $o\left(n\right)$ vertices
of $F_{i}''$. Let $Y$ be the set of uncovered vertices.

Now, it suffices to find a perfect matching in $F_{i}'\left[Y\cup X\cup Z'\right]$.
By the richness of $Z'$, in $F_{i}'$ we can greedily find a matching
of $\left|Y\right|$ edges each with a vertex in $Y$ and two
vertices in $Z'$. We can then greedily augment this matching with
edges inside $Z'$ until there are $\left|Z'\right|\left(1-\xi\right)=\left|Z\right|/2$
vertices of $Z'$ uncovered. Finally, we can use the absorbing structure
to finish the perfect matching.
\end{proof}

\section{\label{sec:random-good}Goodness in random Steiner triple systems}

Let ${\boldsymbol S}$ be a uniformly random ordered {\bf STS}$(n)$,
and let $N={n \choose 2}/3=\left(1+o\left(1\right)\right)n^{2}/6$.
We will show that ${\boldsymbol S}_{\alpha N}$ is a.a.s.\ $(\alpha,o(1))$-good as long as $\alpha$
is a sufficiently small constant. This will suffice to prove \cref{conj:random-good},
because we can partition ${\boldsymbol S}$ into $r$ partial Steiner triple
systems which each have the same distribution as ${\boldsymbol S}_{N/r}$ and then take a union bound.

\subsection{Almost-regularity}

The almost-regularity property of goodness follows immediately
from \cref{lem:triangle-removal-transfer}, taking $\mathcal{Q}={\mathcal O}_{\alpha N}^{\varepsilon,1}$
and observing that every $S\in{\mathcal O}_{\alpha N}^{\varepsilon,1}$ has
the required property. (We could also give a direct proof by considering
a random ordering of the edges of a Steiner triple system; see
\cite[Section~2.1]{Kwa16}).

\subsection{Upper-quasirandomness\label{subsec:upper-uniformity}} 
In this subsection we prove that ${\boldsymbol S}_{\alpha N}$ a.a.s. satisfies the
quasirandomness property of goodness. Consider $q\in{\mathbb N}$
(which we will treat as a sufficiently large constant).

Let $G\sim \operatorname{H}^3(n,p)$ with $p=\left(\alpha/q\right)\left(1+10\left(\alpha/q\right)\right)/n$.
By the Chernoff bound, with probability $1-e^{-\Omega\left(n^{2}\right)}$
our random $3$-graph $G$ has the property that $e_G\left(X,Y,Z\right)\le\left(\alpha/q\right)\left|X\right|\left|Y\right|\left|Z\right|/n+11\left(\alpha/q\right)^{2}n^{2}$
for every triple of disjoint vertex subsets $X,Y,Z$. Denote this
property by $\mathcal{P}$. By \cref{lem:bite-transfer-new} (with
$S=\emptyset$) and \cref{lem:triangle-removal-transfer} (with no
conditioning; that is, $\mathcal{Q}={\mathcal O}_{\alpha N}\cup\left\{ *\right\} $),
$\mathcal{P}$ also holds a.a.s. in ${\boldsymbol S}_{\alpha N/q}$. By symmetry,
in fact it holds a.a.s. in ${\boldsymbol S}^{(\ell)}:={\boldsymbol S}_{\alpha \ell N/q}\backslash{\boldsymbol S}_{\alpha\left(\ell-1\right)N/q}$
for each $\ell\le q$.

So, a.a.s. for every triple of vertex subsets $X,Y,Z$ in ${\boldsymbol S}_{\alpha N}$,
we have $$e_{{\boldsymbol S}_{\alpha N}}\left(X,Y,Z\right)=\sum_{\ell\le q}e_{{\boldsymbol S}^{(\ell)}}\left(X,Y,Z\right)\leq \alpha\left|X\right|\left|Y\right|\left|Z\right|/n+11\left(\alpha^{2}/q\right)n^{2}.$$
Since $q$ could have been arbitrarily large, the desired result follows.

\subsection{Embedding Absorbers}

To finish the proof of \cref{conj:random-good}, we need to prove that
the robust absorber-resilience property of goodness holds a.a.s.
in ${\boldsymbol S}_{\alpha N}$. To do this we use \cref{lem:triangle-removal-transfer}
in its full generality, conditioning on the almost-regularity of the
first few steps of the triangle removal process.

Let $\alpha'=\alpha/2$, let $a$ be small enough for \cref{lem:triangle-removal-transfer},
and let 
\[
\mathcal{Q}=\left\{ *\right\} \cup\left\{ S\in{\mathcal O}_{\alpha N}:S_{\alpha'N}\in{\mathcal O}_{\alpha'N}^{n^{-a},1}\right\} \supseteq{\mathcal O}_{\alpha N}^{n^{-a},1}
\]
The plan is to use \cref{lem:bite-transfer-new}
and \cref{thm:KLR} to show that for any $S\in{\mathcal O}_{\alpha'N}^{n^{-a},1}$,
the triangle-removal process ${\operatorname{R}}\left(S,\alpha'N\right)$ is extremely
likely to produce a partial Steiner triple system which has certain
properties which make it easy to find absorbers. We will then be able
to use \cref{lem:triangle-removal-transfer} to prove that ${\boldsymbol S}_{\alpha N}$
is likely to have these properties as well. In combination with some
much simpler facts about concentration of degrees in ${\boldsymbol S}_{\alpha N}$,
this will allow us to deduce the desired robust absorber-resilience property.
Fix some $\beta>0$ that is very small compared to $\alpha$, and
for every subset of vertices $W\subseteq\left[n\right]$ fix an equipartition
$W=\pi_{A}\left(W\right)\cup\pi_{B}\left(W\right)$. Our main objective
in this subsection is to prove the following claim.
\begin{claim}
\label{claim:absorbers-in-subgraphs}Consider any $S\in{\mathcal O}_{\alpha'N}^{n^{-a},1}$,
and let $\boldsymbol G\sim \operatorname{H}^3(n,p)$ with $p=\alpha'\left(1+10\alpha'\right)/n$.
If $\beta>0$ is sufficiently small, then for any subset $W$ of $w\ge\beta n$
vertices, the following holds with probability $1-e^{-\Omega\left(n^{2}\right)}$.
Consider any spanning subgraphs $G'\subseteq \boldsymbol G\left[W\right]$
and $S'\subseteq S\left[W\right]$ such that
\begin{enumerate}
\item $S'$ has minimum vertex degree at least $0.98\left(\alpha'w^{2}/n^{2}\right)\left(n/2\right)$,
and
\item in $G'$, for any vertex $v\in W$ and every $\beta^3 n$-vertex subset
$U$,
\[
\deg_{\pi_{A}\left(W\right)\backslash U}\left(v\right),\deg_{\pi_{B}\left(W\right)\backslash U}\left(v\right)\ge0.98\left(\alpha'\left(\frac{w/2}{n}\right)^{2}\right)\frac{n}{2}.
\]
\end{enumerate}
Then for any vertices $x,y,z\in W$, there is an absorber in $S'\cup G'$
rooted at $x,y,z$.
\end{claim}

Before proving \cref{claim:absorbers-in-subgraphs} we show how it
completes the proof of \cref{conj:random-good}. We break down this
deduction into some relatively simple claims. First, in ${\boldsymbol S}_{\alpha N}$
the degrees into various subsets are typically quite well-behaved. Let $\gamma_w=0.0001(w/n)^2\alpha$.
\begin{claim}
\label{claim:almost-regular}A.a.s. ${\boldsymbol S}_{\alpha N}$ has the property
that for any $w\ge\beta n$ and at least $\left(1-\exp\left(-10^{-8}(w/n)^4\alpha^2 n\right)\right)\binom{n}{w}$
of the $w$-vertex subsets $W$:
\begin{enumerate}
\item all the vertex degrees in ${\boldsymbol S}_{\alpha'N}\left[W\right]$ are $\left(\alpha'w^{2}/n^{2}\right)\left(n/2\right)\pm\gamma_w n$,
and
\item in $\left({\boldsymbol S}_{\alpha N}\backslash{\boldsymbol S}_{\alpha'N}\right)\left[W\right]$
every vertex $v$ has degree $\left(\alpha'\left(w/2\right)^{2}/n^{2}\right)\left(n/2\right)\pm\gamma_w n$
into $\pi_{A}\left(W\right)$ and into $\pi_{B}\left(W\right)$, and
there are $2\left(\alpha'\left(w/2\right)^{2}/n^{2}\right)\left(n/2\right)\pm\gamma_w n$
edges containing $v$, a vertex in $\pi_{A}\left(W\right)$ and
a vertex in $\pi_{B}\left(W\right)$.
\end{enumerate}
\end{claim}

\begin{proof}
Let $w\ge\beta n$ and consider a particular $w$-vertex set $W$.
Observe that randomly reordering the edges and vertices of ${\boldsymbol S}$
does not change its distribution, so by a concentration inequality
for the hypergeometric distribution (see for example \cite[Theorem~2.10]{JLR00}) and the union bound,
with probability at least $1-e^{-(3/2)\gamma_w^2 n}$ the desired properties
hold for $W$. By Markov's inequality, a.a.s. for every $w\ge\beta n$
the number of $W$ for which the properties fail to hold is at most
$e^{-\gamma_w^2 n}\binom{n}{w}=\exp\left(-10^{-8}(w/n)^4\alpha^2 n\right)\binom{n}{w}$.
\end{proof}
Next, the following lemma summarises how to use \cref{lem:triangle-removal-transfer}
and \cref{lem:bite-transfer-new} to turn \cref{claim:absorbers-in-subgraphs}
into a fact about random Steiner triple systems.
\begin{claim}
\label{claim:absorbers-in-subsystems}A.a.s. ${\boldsymbol S}_{\alpha N}$
has the following property, provided $\beta$ is sufficiently small.
Consider any $w\ge\beta n$ and any $w$-vertex subset $W$, and consider
spanning subgraphs $S''\subseteq\left({\boldsymbol S}_{\alpha N}\backslash{\boldsymbol S}_{\alpha'N}\right)\left[W\right]$
and $S'\subseteq{\boldsymbol S}_{\alpha'N}\left[W\right]$ such that
\begin{enumerate}
\item $S'$ has minimum vertex degree at least $0.98\left(\alpha'w^{2}/n^{2}\right)\left(n/2\right)$,
and
\item in $S''$, for any vertex $v\in W$ and every $\beta^3 n$-vertex subset
$U$, we have 
\[
\deg_{\pi_{A}\left(W\right)\backslash U}\left(v\right),\deg_{\pi_{B}\left(W\right)\backslash U}\left(v\right)\ge0.98\left(\alpha'\left(\frac{w/2}{n}\right)^{2}\right)\frac{n}{2}.
\]
\end{enumerate}
Then for any vertices $x,y,z\in W$, there is an absorber in $S'\cup S''$
rooted at $x,y,z$.
\end{claim}

\begin{proof}
Note that the property in \cref{claim:absorbers-in-subgraphs} is
a monotone decreasing property of ${\boldsymbol G}$ (depending on $S$). Then,
combine \cref{lem:bite-transfer-new,lem:triangle-removal-transfer,claim:absorbers-in-subgraphs}
(with $\mathcal{Q}$ as in the beginning of this subsection).
\end{proof}
Now, given \cref{claim:almost-regular} and \cref{claim:absorbers-in-subsystems}
it is fairly immediate to deduce that ${\boldsymbol S}_{\alpha N}$ a.a.s.
has the desired absorber-resilience property, as follows. Suppose
that ${\boldsymbol S}_{\alpha N}$ satisfies the conclusions of \cref{claim:almost-regular}
and \cref{claim:absorbers-in-subsystems}, and consider $w\ge\beta n$.
Suppose for some $w$-vertex subset $W$ that ${\boldsymbol S}_{\alpha'N}\left[W\right]$
and $\left({\boldsymbol S}_{\alpha N}\backslash{\boldsymbol S}_{\alpha'N}\right)\left[W\right]$
both satisfy the degree conditions in \cref{claim:almost-regular}
(this is true for at least $\left(1-\exp\left(-10^{-8}(w/n)^4\alpha^2 n\right)\right)\binom{n}{w}$
of the choices of $W$). We want to show that ${\boldsymbol S}_{\alpha N}[W]$
is $D$-absorber-resilient, for $D=0.999\left(\alpha w^{2}/n^{2}\right)\left(n/2\right)$.
Consider a subgraph $S\subseteq{\boldsymbol S}_{\alpha N}\left[W\right]$
with minimum degree $D$.
Observe that by the choice of $\gamma_w$, the hypergraph $S'=S\cap\boldsymbol{S}_{\alpha'N}$ consisting of the edges of $S$ that appear in $\boldsymbol{S}_{\alpha'N}$ has minimum degree at least 
\[
0.999\left(\alpha\frac{w^{2}}{n^{2}}\right)\frac{n}{2}-4\left(\left(\alpha'\left(\frac{w/2}{n}\right)^{2}\right)\frac{n}{2}\pm\gamma_w n\right)\ge0.98\left(\alpha'\frac{w^{2}}{n^{2}}\right)\frac{n}{2}.
\]
Next, let $S''=S\setminus\boldsymbol{S}_{\alpha'N}$, consider any vertex $v$ and any set of $U$ of $\beta^3 n$ other vertices of $W$. Since
$S$ is a partial Steiner triple system there are at most $\beta^3 n$
edges of $S$ involving both $v$ and $U$, so for sufficiently small $\beta$, in $S''$ we have
\begin{align*}
\deg_{\pi_{A}\left(W\right)\backslash U}\left(v\right)&\ge0.999\left(\alpha'\frac{w^{2}}{n^{2}}\right)\frac{n}{2}-\left(\left(\alpha'\frac{w^{2}}{n^{2}}\right)\left(n/2\right)\pm\gamma_w n\right)-3\left(\left(\alpha'\left(\frac{w/2}{n}\right)^{2}\right)\frac{n}{2}\pm\gamma_w n\right)-\beta^3 n\\
&\ge0.98\left(\alpha'\left(\frac{w/2}{n}\right)^{2}\right)\frac{n}{2},
\end{align*}
and similarly $\deg_{\pi_{B}\left(W\right)\backslash U}\left(v\right)\ge0.98\left(\alpha'\left(w/2\right)^{2}/n^{2}\right)\left(n/2\right)$.
So, in $S'$ there is an absorber on every triple of vertices, concluding
the proof of $D$-absorber-resilience.

It remains to prove \cref{claim:absorbers-in-subgraphs}. As in the
statement of \cref{claim:absorbers-in-subgraphs}, consider any $S\in{\mathcal O}_{\alpha'N}^{n^{-a},1}$
and let $\boldsymbol G\sim \operatorname{H}^3(n,p)$ with $p=\alpha'\left(1+10\alpha'\right)/n$.
The proof of \cref{claim:absorbers-in-subgraphs} will reduce to another
claim about a family of auxiliary graphs obtained by contracting edges
of subgraphs of $\boldsymbol G$. For a vertex set $W$ containing a vertex
$v$, let $H_{S\left[W\right]}\left(v\right)$ be the set of all edges
containing $v$ in $S[W]$.
\begin{defn}
\label{def:contracted}Consider any subset $W$ of $w\ge\beta n$
vertices of $\boldsymbol G$, fix any $x,y,z\in W$, and consider any subsets
$H_{x}\subseteq H_{S\left[W\right]}\left(x\right),H_{y}\subseteq H_{S\left[W\right]}\left(y\right),H_{z}\subseteq H_{S\left[W\right]}\left(z\right)$
each with $0.98\left(\alpha' w^{2}/n^{2}\right)\left(n/2\right)$ edges.
Let $\boldsymbol G\left(W,x,y,z,H_{x},H_{y},H_{z}\right)$ be a 3-graph obtained
from $\boldsymbol G\left[W\backslash\left\{ x,y,z\right\} \right]$ by doing
the following.

\begin{enumerate}
\item First, for $k=\beta^{4}n$, choose distinct vertices 
\[
a_{1}^{x},\dots,a_{k}^{x},b_{1}^{x},\dots,b_{k}^{x},a_{1}^{y},\dots,a_{k}^{y},b_{1}^{y},\dots,b_{k}^{y},a_{1}^{z},\dots,a_{k}^{z},b_{1}^{z},\dots,b_{k}^{z}
\]
such that for each $0\le i\le k$, we have $\left\{ x,a_{i}^{x},b_{i}^{x}\right\} \in H_{x},\left\{ y,a_{i}^{y},b_{i}^{y}\right\} \in H_{y},\left\{ z,a_{i}^{z},b_{i}^{z}\right\} \in H_{z}$.
If $\beta$ is sufficiently small this can be done greedily, recalling
that the edges of $H_{x}$ (respectively, of $H_{y}$ or of $H_{z}$)
intersect only at $x$ (respectively, only at $y$ or at $z$).
\item Let $W_{A}$ (respectively $W_{B}$) be the subset of $\pi_{A}\left(W\right)$
(respectively $\pi_{B}\left(W\right)$) obtained by removing $x,y,z$ and all $3k$ of the vertices chosen in the first step.
\item Then, for each $v\in\left\{ x,y,z\right\} $, delete all edges
containing a vertex of $H_{v}$, except those containing two vertices
from $W_{A}$ and some $a_{i}^{v}$, or two vertices from $W_{B}$
and some $b_{i}^{v}$.
\item Finally, for each $v\in\left\{ x,y,z\right\} $, contract each pair
$\left\{ a_{i}^{v},b_{i}^{v}\right\} $ to a single vertex. Let $U_{v}$
be the set of newly contracted vertices.
\end{enumerate}
\end{defn}

Now, to prove \cref{claim:absorbers-in-subgraphs} it suffices to prove
the following claim (observe that $2\left|U_{x}\right|+2\left|U_{y}\right|+2\left|U_{z}\right|+3\le \beta^3 n$ for sufficiently small $\beta$, and taking the union bound over all
choices of $W,x,y,z,H_{x},H_{y},H_{z}$ costs us a factor of only
$e^{O\left(n\right)}$).
\begin{claim}
\label{claim:technical-absorbers-in-subgraphs}For any $W,x,y,z,H_{x},H_{y},H_{z}$
as in \cref{def:contracted}, ${\boldsymbol G}':={\boldsymbol G}\left(W,x,y,z,H_{x},H_{y},H_{z}\right)$
has the following property with probability $1-e^{-\Omega\left(n^{2}\right)}$.
For any spanning subgraphs 
\[
G_{A}\subseteq{\boldsymbol G}'\left[W_{A}\cup U_{x}\cup U_{y}\cup U_{z}\right],\quad G_{B}\subseteq{\boldsymbol G}'\left[W_{B}\cup U_{x}\cup U_{y}\cup U_{z}\right]
\]
with minimum vertex degree at least $0.98\left(\alpha\left(w/2\right)^{2}/n^{2}\right)\left(n/2\right)$,
there are vertices $x'\in U_{x}$, $y'\in U_{y}$, $z'\in U_{z}$
such that in both $G_{A}$ and $G_{B}$ there is a sub-absorber rooted
on $x',y',z'$.
\end{claim}

We will prove \cref{claim:technical-absorbers-in-subgraphs} with the
sparse regularity lemma (\cref{lem:sparse-regularity}) and \cref{thm:KLR}.
Before doing this, we prove some final auxiliary lemmas. First, we
need to know that absorbers are suitably sparse to apply \cref{thm:KLR}.
Recall the definition of 3-density from \cref{def:3-density}.
\begin{lem}
\label{lem:m3-contracted-absorber}Consider an absorber and contract each rooted edge to a single vertex. Call the resulting
3-graph $H$ a \emph{contracted absorber}. ($H$ is equivalently obtained
by gluing together two sub-absorbers on their rooted vertices). Then
$H$ has $m_{3}\left(H\right)<1$.
\end{lem}

\begin{proof}
Note that $H$ has maximum degree 2. Let $H'$ be a subgraph of $H$
with $q$ vertices of degree $2$. Then, $3e\left(H'\right)\le v\left(H'\right)+q$,
and it follows that we can only have $\left(e\left(H'\right)-1\right)/\left(v\left(H'\right)-3\right)\ge1$
if $q\ge2v\left(H'\right)-6$. This is impossible if $v(H')>6$, so it
suffices to consider subgraphs $H'$ with up to 6 vertices. It then
suffices to check by hand that every pair of edges spans at least
5 vertices, and every triple of edges spans at least 7 vertices.
\end{proof}
Second, the sparse regularity lemma will give us a very dense cluster
3-graph, and we will need to be able to find absorbers in such a 3-graph.
\begin{lem}
\label{lem:sub-absorber-in-dense}Let $F$ be an $n$-vertex 3-graph
with degrees at least $0.96\binom{n}{2}$, having an identified subset
$U$ of at most $0.001n$ vertices. If $n$ is sufficiently large, then
for any $x,y,z,\in U$ there is a sub-absorber rooted at $x,y,z$
in which the only edges involving vertices of $U$ are the rooted
ones.
\end{lem}

\begin{proof}Let $V=V\left(F\right)$ and consider any $x,y,z\in U$. Every vertex
has degree at least $0.96\binom{n}{2}-0.001n^{2}$ into $V\backslash U$, so there are at least $\left(0.96\binom{n}{2}-0.001n^2\right)^3-o(n^6)$ choices of three disjoint edges each containing one of $x$,$y$ and $z$ and two vertices in $V\setminus U$. For each such choice of three edges, there is a pair of edges in $U$ whose presence would yield a suitable sub-absorber, and a pair of edges in $F\left[V\backslash U\right]$ can contribute in this way to at most $(3!)^2$ sub-absorbers. It follows that there are at least $\left(0.96\binom{n}{2}-0.001n^2\right)^3/(3!)^2-o(n^6)$ pairs of edges whose presence would yield a sub-absorber. But note that all but at most $(1-0.96^{2})\binom{n}{3}^{2}$ of the possible pairs of edges
are present in $F$, and $\left(0.96/2!-0.001\right)^3/(3!)^2>(1-0.96^2)(1/3!)^2$, so for large $n$, there must be a sub-absorber as desired.
\end{proof}
Now we can finally prove \cref{claim:technical-absorbers-in-subgraphs},
completing the proof of \cref{conj:random-good}.
\begin{proof}[Proof of \cref{claim:technical-absorbers-in-subgraphs}]
Consider all the triples of vertices intersecting $U_{x}\cup U_{y}\cup U_{z}$,
except those with one vertex in $U_{x}\cup U_{y}\cup U_{z}$ and two
vertices in $W_{A}$, or one vertex in $U_{x}\cup U_{y}\cup U_{z}$
and two vertices in $W_{B}$. None of these triples are edges
in ${\boldsymbol G}'$. Let ${\boldsymbol G}_{\text{extra}}$ be a random 3-graph, independent
from ${\boldsymbol G}'$, where each of these triples is included with probability
$p:=\alpha'\left(1+10\alpha'\right)/n$ independently. Now ${\boldsymbol G}'\cup{\boldsymbol G}_{\text{extra}}$
is a standard binomial random 3-graph on the vertex set $U_{x}\cup U_{y}\cup U_{z}\cup W_{A}\cup W_{B}$
where each edge is present with probability $p$. The only purpose
of the edges in ${\boldsymbol G}_{\text{extra}}$ is to put us in the
setting for \cref{thm:KLR}; we will not actually use these edges
for anything.

Let $H$ be a contracted absorber, as defined in \cref{lem:m3-contracted-absorber}.
By \cref{thm:KLR} there is $\varepsilon>0$ such that with probability
$1-e^{-\Omega\left(n^{2}\right)}$ the random 3-graph ${\boldsymbol G}'\cup{\boldsymbol G}_{\text{extra}}$
has the property that for any $n'=\Omega\left(n\right)$ and any $m\ge0.001p\left(n'\right)^{3}$,
every subgraph $G''\subseteq{\boldsymbol G}'\cup{\boldsymbol G}_{\text{extra}}$ in $\mathcal{G}\left(H,n',m,p,\varepsilon\right)$
has $\#_H(G')>0$. Also, by the Chernoff bound (basically
as in \cref{subsec:upper-uniformity}), ${\boldsymbol G}$ is $\left(p,\beta\right)$-upper-quasirandom
with probability $1-e^{-\Omega\left(n^{2}\right)}$.

Assuming that the above two properties hold, apply \cref{lem:sparse-regularity,lem:respect-partition}
to $G_{A}\cup G_{B}$ with regularity parameter $\varepsilon>0$ which
is small even compared to $\beta$, and with $t_{0}\ge1/\varepsilon$,
to obtain a partitioned $t$-vertex cluster 3-graph $\mathcal{R}$
with threshold $0.001$. Let $\mathcal{U}_{x},\mathcal{U}_{y},\mathcal{U}_{z},\mathcal{W}_{A},\mathcal{W}_{B}$
be the sets of clusters fully contained in $U_{x},U_{y},U_{z},W_{A},W_{B}$.
By \cref{lem:transfer-cluster-graph-partition}, in both $\mathcal{R}_{A}:=\mathcal{R}\left[\mathcal{W}_{A}\cup\mathcal{U}_{x}\cup\mathcal{U}_{y}\cup\mathcal{U}_{z}\right]$
and in in $\mathcal{R}_{B}:=\mathcal{R}\left[\mathcal{W}_{B}\cup\mathcal{U}_{x}\cup\mathcal{U}_{y}\cup\mathcal{U}_{z}\right]$,
all but $\sqrt{\varepsilon}t$ vertices have degree at least $0.97\binom{t/2}{2}$.
Delete at most $2\sqrt{\varepsilon}t$ vertices to obtain induced
subgraphs $\mathcal{R}_{A}',\mathcal{R}_{B}'$ with minimum degree
at least $0.96\binom{t/2}{2}$.

Now, fix clusters $V_{x}\in\mathcal{U}_{x},V_{y}\subseteq\mathcal{U}_{y},V_{z}\subseteq\mathcal{U}_{z}$
which appear in both $\mathcal{R}_{A}',\mathcal{R}_{B}'$, and apply
\cref{lem:sub-absorber-in-dense} to $\mathcal{R}_{A}'$ and to $\mathcal{R}_{B}'$
to find sub-absorbers in the cluster graph rooted at $V_{x},V_{y},V_{z}$.
These two sub-absorbers comprise a contracted absorber $\mathcal{A}$
in $\mathcal{R}$, so the edges between the clusters in $V\left(\mathcal{A}\right)$
give us a subgraph of ${\boldsymbol G}'$ in $\mathcal{G}\left(H,n',m,p,\varepsilon\right)$,
for $n'=n/t\ge n/T$ and $m\ge0.001p\left(n'\right)^{3}$. It follows
that we can find a canonical copy of a contracted absorber $H$, giving
us a sub-absorber in $G_{A}$ and in $G_{B}$ rooted on the same three
vertices.
\end{proof}

\section{\label{sec:concluding}Concluding remarks}

There are many interesting further directions of research regarding
random Steiner triple systems. Most obviously, \cref{conj:kirkman}
is still open, though we imagine that an exact result would be quite
difficult to prove. Perhaps a good starting point would be the methods
of K\"uhn and Osthus~\cite{KO13}, and Knox, K\"uhn and Osthus~\cite{KKO15}
for perfectly packing Hamilton cycles in random graphs and random
tournaments.

A second interesting direction is to study the discrepancy of random
Steiner triple systems. Combining the ideas in \cref{subsec:upper-uniformity}
and \cite[Section~5.1.2]{Kwa16}, we have essentially proved the following
theorem bounding the discrepancy of a random Steiner triple system,
which may be of independent interest.
\begin{thm}
\label{thm:discrepancy}Let ${\boldsymbol S}$ be a uniformly random order-$n$
Steiner triple system. Then ${\boldsymbol S}$ a.a.s.\ satisfies the following
property. For every triple of vertex subsets $X,Y,Z$, we have $e\left(X,Y,Z\right)=\left|X\right|\left|Y\right|\left|Z\right|/n+o\left(n^{2}\right)$.
\end{thm}

We remark that an analogous theorem for Latin squares (with a stronger
error term) was proved by Kwan and Sudakov~\cite{KS16}. See also
the related conjectures in \cite{LL16}. It would be very interesting
if one could substantially improve the error term $o\left(n^{2}\right)$;
we imagine the correct order of magnitude is $O\left(\sqrt{\left|X\right|\left|Y\right|\left|Z\right|}\right)$,
but a proof of this would require substantial new ideas.

Next, another interesting direction concerns containment and enumeration
of subgraphs. Using \cref{thm:KLR} and the sparse regularity lemma
(\cref{lem:sparse-regularity}) in combination with \cref{lem:triangle-removal-transfer}
and \cref{lem:bite-transfer-new}, it is straightforward to prove the
following result.
\begin{thm}
Let $H$ be a $3$-graph with $m_{3}\left(H\right)<1$. Then there
is $\xi=\xi\left(H\right)>0$ such that a uniformly random {\bf STS}$(n)$ a.a.s.\ contains at least $\xi n^{v\left(H\right)-e\left(H\right)}$
copies of $H$.
\end{thm}

We imagine that actually the number of (labelled) copies of $H$ should
a.a.s.\ be $\left(1\pm o\left(1\right)\right)n^{v\left(H\right)-e\left(H\right)}$,
but it is less obvious how to prove this. In particular, due to the
``infamous upper tail'' issue (see \cite{JR02}) and the fact that
\cref{lem:triangle-removal-transfer} only works with properties that
hold with probability extremely close to 1, any kind of upper bound
on subgraph counts would require new ideas. We also remark that a
lower bound was proved by Simkin~\cite{Sim18} in the special case
where $H$ is a \emph{Pasch configuration}, using ideas from \cite{Kwa16}.

Another interesting question (also mentioned in \cite{Kwa16}) is
whether a random Steiner triple system typically contains a Steiner
triple subsystem on fewer vertices. McKay and Wanless~\cite{MW99}
proved that almost all Latin squares have many small Latin subsquares
(see also \cite{KS16}), but it was conjectured by Quackenbush~\cite{Qua80}
that most Steiner triple systems do not have proper subsystems. By
comparison with a binomial random 3-graph, it seems likely that this
conjecture is actually false, but it seems that substantial new ideas
would be required to prove or disprove it.

Finally, one might try to generalise from Steiner triple systems to
other classes of designs. It seems that the arguments in this paper
should generalise in a straightforward fashion to Latin squares, proving
that a random order-$n$ Latin squares a.a.s.\ has $n-o\left(n\right)$
disjoint transversals (see \cite{Kwa16} for a definition of Latin
squares, a discussion of how the methods in \cref{sec:random-STS}
generalise to random Latin squares, and a discussion of the significance
of transversals in Latin squares). It was actually conjectured by
van Rees~\cite{Ree90} that a random order-$n$ Latin square typically does
\emph{not} have a decomposition into $n$ disjoint transversals, though
Wanless and Webb~\cite{WW06} observed that numerical observations
seem more in line with the Latin squares analogue of \cref{conj:kirkman}.

Also, a \emph{$\left(q,r,\lambda\right)$-design} ($q>r$) of order
$n$ is a $q$-uniform hypergraph on the vertex set $\left[n\right]$ such
that every $r$-set of vertices is included in exactly $\lambda$
edges. A $\left(q,r\right)$\emph{-Steiner system} is a $\left(q,r,1\right)$-design
(so, a Steiner triple system is a $\left(3,2,1\right)$-design or
equivalently a $\left(3,2\right)$-Steiner system, and a $d$-regular
graph is a $\left(2,1,d\right)$-design). We expect that it should
be fairly routine to adapt the definition of an absorber in the obvious
way to prove that almost all $\left(q,r,\lambda\right)$-designs
have a decomposition of almost all their edges into perfect matchings.
As for Steiner triple systems, a design is said to be resolvable if it admits
a decomposition into perfect matchings, and the general problem of
whether resolvable block designs exist was only very recently solved
by Keevash~\cite{Kee18}.

However, note that a $3$-uniform perfect matching is actually a $\left(3,1\right)$-Steiner
system, so as a sweeping generalisation of \cref{conj:kirkman} we
might ask for which $r'\le r$ and $\lambda'\le\lambda$ do $\left(q,r,\lambda\right)$-designs
typically admit a decomposition into spanning $\left(q,r',\lambda'\right)$-designs
of the same order. We note that in the case of regular graphs a much
stronger phenomenon occurs: there is a sense in which a random $\left(d_{1}+d_{2}\right)$-regular
graph is ``asymptotically the same'' as a random $d_{1}$-regular
graph combined with a random $d_{2}$-regular graph (provided $d_1+d_2>2$; see \cite[Section~9.5]{JLR00}).

{\bf Acknowledgements.} We would like to thank the anonymous referee for their careful reading of the manuscript and their valuable comments.



\begin{appendices}
\crefalias{section}{appsec}

\section{Adapting the proof of the K\L R Conjecture to linear hypergraphs\label{sec:KLR-proof}}

In their paper~\cite{CGSS14} on the K\L R conjecture for graphs,
Conlon, Gowers, Samotij and Schacht explictly mentioned that their
methods should extend to hypergraphs, and that the generalisation
should be particularly simple in the case of linear hypergraphs. Actually,
the generalisation is so simple that we can describe in this short
appendix exactly what changes to make to their proof of their Theorem~1.6~(i)
(appearing in Section~2 of their paper), to turn it into a proof
of \cref{thm:KLR}. All theorem/lemma/section references are with respect
to \cite{CGSS14}, unless noted otherwise.
\begin{itemize}
\item Change the notation ``$G(H)$'' to ``$\#_H(G)$'' (we changed this notation to avoid confusion with other notation in the paper).
\item Change every instance of ``graph'' to ``$r$-graph'' and every
instance of ``bipartite graph'' to ``$r$-partite $r$-graph''.
\item Change every instance of ``$m_{2}$'' to ``$m_{r}$''.
\item Change every instance of ``$n^{2}$'' to ``$n^{r}$'', every instance
of ``$n^{v\left(H\right)-2}$'' to ``$n^{v\left(H\right)-r}$'',
and every instance of ``$N^{2}$'' to ``$N^{r}$''.
\item The definition of $\left(\varepsilon,d\right)$-lower-regularity preceding
the statement of Theorem~2.1 should be changed in the obvious way:
an $r$-partite $r$-graph between sets $V_{1},\dots,V_{r}$ is $\left(\varepsilon,d\right)$-lower-regular
if, for every $V_{1}'\subseteq V_{1},\dots,V_{r}'\subseteq V_{r}$
with $\left|V_{i}'\right|\ge\varepsilon\left|V_{i}\right|$, the density
$d\left(V_{1}',\dots,V_{r}'\right)$ of edges between $V_{1}',\dots,V_{r}'$
satisfies $d\left(V_{1}',\dots,V_{r}'\right)\ge d$.
\item In the statement of Theorem~2.1, change ``let $H$ be an arbitrary
graph'' to ``let $H$ be an arbitrary linear $r$-graph''.
\item At the beginning of the deduction of Theorem~1.6~(i) from Theorem~2.1,
change the observation ``$m_{2}\left(H\right)\ge1$'' to ``$m_{r}\left(H\right)\ge1/\left(r-1\right)$''.
\item In the second displayed equation in the proof of Theorem~1.6~(i),
change ``$e_{G_{N,p}}\left(W_{i},W_{j}\right)\ge2pn^{2}$ for some
$ij\in E\left(H\right)$'' to ``$e_{G_{N,p}}\left(W_{i_{1}},\dots,W_{i_{r}}\right)\ge2pn^{r}$
for some $\left\{ i_{1},\dots,i_{r}\right\} \in E\left(H\right)$''.
\item A one-sided counting lemma is presented without proof as Lemma ~2.4.
The corresponding generalisation for linear hypergraphs (with ``every
graph $H$'' replaced with ``every linear $r$-graph $H$'') follows
from essentially the same proof as \cite[Lemma~10]{KNRS10} (see also
\cite[Lemma~22]{CHPS12}).
\item In Section~2.2, change the definition of $C$ from ``$32RL^{R}C'/\left(\varepsilon^{2}d\right)$''
to $\text{\textquotedblleft}16rRL^{R}C'/\left(\varepsilon^{2}d\right)\text{\textquotedblright}$,
and in the proof of Claim~2.5 change the last display from ``$e\left(H\right)\cdot2^{2n}\exp\left(-\varepsilon^{2}dp_{s}n^{2}/16\right)$''
to ``$e\left(H\right)\cdot2^{rn}\exp\left(-\varepsilon^{2}dp_{s}n^{r}/16\right)$''.
\item In Section~2.2.2, the definition of $Z_{s}$ should be generalised
in the obvious way:
\[
Z_{s}=\left\{ e\in G\left(V_{i_{r}},\dots,V_{i_{r}}\right):\deg_{H''}\left(e,G,G_{s}'\right)\ge\frac{\xi'}{2}p_{s}^{e\left(H''\right)}n^{v\left(H\right)-r}\right\} ,
\]
where $\left\{ i_{1},\dots,i_{r}\right\} $ is the edge of $H'$ that
is missing in $H''$, and $V_{i_{1}},\dots,V_{i_{r}}$ are the subsets
of $V\left(G\right)$ corresponding to the vertices $i_{1},\dots,i_{r}$.
\item In Case~1 of the proof of Claim~2.6, change ``there exist sets
$X_{i}\subseteq V_{i}$ and $X_{j}\subseteq V_{j}$ with $\left|X_{i}\right|,\left|X_{j}\right|\ge\varepsilon n$''
to ``there exist sets $X_{i_{1}}\subseteq V_{i_{1}},\dots X_{i_{r}}\subseteq V_{i_{r}}$
with $\left|X_{i_{1}}\right|,\dots,\left|X_{i_{r}}\right|\ge\varepsilon n$''.
Also, change all instances of ``$\left(X_{i},X_{j}\right)$'' to
``$\left(X_{i_{1}},\dots,X_{i_{r}}\right)$'' and change all instances
of ``$\left|X_{i}\right|\left|X_{j}\right|$ to $\left|X_{i_{1}}\right|\dots\left|X_{i_{r}}\right|$''.
\item In Case~2 of the proof of Claim~2.6, change ``for every $i',j'\in E\left(H\right)$
and every pair of sets $W_{i'}\subseteq V_{i'}$ and $W_{j'}\subseteq V_{j'}$\ensuremath{\in}
with $\left|W_{i'}\right|\ge\varepsilon n$ and $\left|W_{j'}\right|\ge\varepsilon n$''
to ``for every $i_{1}',\dots,i_{r}'\in E\left(H\right)$ and any
choice of $W_{i_{1}'}\subseteq V_{i_{1}'},\dots,W_{i_{r}'}\subseteq V_{i_{r}'}$
with $\left|W_{i_{1}'}\right|,\dots,\left|W_{i_{r}'}\right|\ge\varepsilon n$''.
Also, change all instances of ``$\left(W_{i'},W_{j'}\right)$''
to ``$\left(W_{i_{1}'},\dots,W_{i_{r}'}\right)$'', all instances
of ``$\left|W_{i'}\right|\left|W_{j'}\right|$'' to ``$\left|W_{i_{1}'}\right|\dots\left|W_{i_{r}'}\right|$'',
and all instances of ``$\left(V_{i},V_{j}\right)$'' to ``$\left(V_{i_{1}},\dots,V_{i_{r}}\right)$''.
\end{itemize}
\end{appendices}
\end{document}